\documentclass[a4paper,leqno]{amsart}
\usepackage{amsfonts}

\usepackage{amssymb,amsthm,amsmath,amsrefs,enumerate,lscape,verbatim}
\usepackage{pstricks}
\usepackage{pst-plot}
\usepackage{pst-math}
\usepackage{pst-node}

\newcounter{mylisti} \newcounter{mylistii}
\newcounter{nest}
\newcommand{\defaultlabel}{}

\newcommand{\bb}{\ensuremath{\mathbb{B}}}

\newcommand{\bn}{\ensuremath{\mathbb N}}
\newcommand{\bp}{\ensuremath{\mathbb P}}

\newcommand{\br}{\ensuremath{\mathbb R}}

\newcommand{\cB}{\ensuremath{\mathcal B}}
\newcommand{\cC}{\ensuremath{\mathcal C}}

\newcommand{\cG}{\ensuremath{\mathcal G}}

\newcommand{\cN}{\ensuremath{\mathcal N}}

\newcommand{\cP}{\ensuremath{\mathcal P}}

\newcommand{\cS}{\ensuremath{\mathcal S}}

\newcommand{\cV}{\ensuremath{\mathcal V}}

\newcommand{\xt}{\ensuremath{\tilde{x}}}
\newcommand{\yt}{\ensuremath{\tilde{y}}}
\newcommand{\vpt}{\ensuremath{\tilde{\varepsilon}}}

\newcommand{\At}{\ensuremath{\tilde{A}}}

\newcommand{\Xt}{\ensuremath{\tilde{X}}}
\newcommand{\Yt}{\ensuremath{\tilde{Y}}}

\newcommand{\vph}{\ensuremath{\hat{\varepsilon}}}

\newcommand{\dist}{\ensuremath{\mathrm{dist}}}

\newcommand{\cof}{\ensuremath{\mathrm{cof}}}

\newcommand{\spa}{{\rm span}}

\newcommand{\supp}{\operatorname{supp}}

\newcommand{\Sz}{\ensuremath{\mathrm{Sz}}}
\newcommand{\Dz}{\ensuremath{\mathrm{Dz}}}

\newcommand{\keq}{\!=\!}

\newcommand{\kin}{\!\in\!}

\newcommand{\kle}{\!<\!}
\newcommand{\kleq}{\!\leq\!}

\newcommand{\kminus}{\!-\!}

\newcommand{\kplus}{\!+\!}

\newcommand{\tn}{|\!|\!|}
\newcommand{\btn}{\big|\!\big|\!\big|}
\newcommand{\bbtn}{\Big|\!\Big|\!\Big|}

\newcommand{\ds}{\displaystyle}

\newcommand{\ie}{\textit{i.e.,}\ }

\newcommand{\sh}{{\scriptstyle\frac12}}
\def\vp{\varepsilon}

\newcommand{\co}{\mathrm{c}_0}

\newtheorem{theorem}{Theorem}[section]
\newtheorem{mainthm}{Theorem}

\newtheorem{lemma}{Lemma}[section]
\newtheorem{proposition}{Proposition}[section]
\newtheorem{corollary}{Corollary}[section]
\newtheorem{claim}{Claim}[section]

\newtheorem{problem}{Problem}[section]

\theoremstyle{definition}

\newtheorem{defn}{Definition}[section]

\theoremstyle{remark}

\newtheorem{rem}{Remark}[section]

\newtheorem{ex}{Example}

\begin{document}

\title[On the geometry of the countably branching diamond graphs]{On the geometry of the countably branching diamond graphs}

\author[Baudier et al.]{Florent~Baudier}
\address{Florent Baudier, Department of Mathematics, Texas A\&M University, College Station, TX 77843-3368, USA}
\email{florent@math.tamu.edu}
\author[]{Ryan Causey}
\address{Ryan Causey, Department of Mathematics, Miami University, Oxford, OH 45056, USA
}
\email{causeyrm@miamioh.edu}
\author[]{Stephen Dilworth}
\address{Steve Dilworth, Department of Mathematics, University of South Carolina, Columbia, SC 29208, USA.}
\email{dilworth@math.sc.edu}
\author[]{Denka Kutzarova}
\address{Denka Kutzarova, Institute of Mathematics, Bulgarian Academy of Sciences, Sofia, Bulgaria and Department of Mathematics, University of Illinois at Urbana-Champaign, Urbana, IL 61801, USA.}
\email{denka@math.uiuc.edu}
\author[]{N. L. Randrianarivony}
\address{N. L. Randrianarivony, Department of Mathematics and Computer Science, Saint Louis University, St. Louis, MO 63103, USA.}
\email{nrandria@slu.edu}
\author[]{Thomas~Schlumprecht}
\address{Thomas Schlumprecht, Department of Mathematics, Texas A\&M University, College Station, TX 77843-3368, USA, and Faculty of Electrical Engineering, 
Czech Technical University in Prague, Zikova 4, 16627, Prague, Czech Republic
}
\email{schlump@math.tamu.edu}
\author[]{Sheng~Zhang}
\address{Sheng Zhang, School of Mathematics, Southwest Jiaotong University, Chengdu, Sichuan 611756, China}
\email{sheng@swjtu.edu.cn}
\date{}

\thanks{The third author was supported by the National Science Foundation under Grant Number DMS-1361461 and by the Workshop in Analysis and Probability at Texas A\&M University in 2015 and 2016.
The fourth author was partially supported by the Bulgarian National Scientific Fund under Grant DFNI-I02/10.
The fifth author was supported by the National Science Foundation under Grant Number DMS-1301591.
The sixth author was supported by the National Science Foundation under Grant Number DMS-146473.
The seventh author was partially supported by the National Science Foundation under Grant Number DMS-1301604.}


\begin{abstract}
In this article, the bi-Lipschitz embeddability of the sequence of countably branching diamond graphs $(D_k^\omega)_{k\in\bn}$ is investigated. In particular it is shown that for every $\vp>0$ and $k\in\bn$, $D_k^\omega$ embeds bi-Lipschiztly with distortion at most $6(1+\vp)$ into any reflexive Banach space with an unconditional asymptotic structure that does not admit an equivalent asymptotically uniformly convex norm. On the other hand it is shown that the sequence $(D_k^\omega)_{k\in\bn}$ does not admit an equi-bi-Lipschitz embedding into any Banach space that has an equivalent asymptotically midpoint uniformly convex norm. Combining these two results one obtains a metric characterization in terms of graph preclusion of the class of asymptotically uniformly convexifiable spaces, within the class of separable reflexive Banach spaces with an unconditional asymptotic structure. Applications to bi-Lipschitz embeddability into $L_p$-spaces and to some problems in renorming theory are also discussed.  
\end{abstract}

\maketitle

\setcounter{tocdepth}{4}
\setcounter{secnumdepth}{4}
\tableofcontents

\section{Introduction}

\subsection{Motivation}
One of the most natural way to grasp the geometry of a metric space is to understand in which metric spaces, in particular which Banach spaces, it does, or it does not, bi-Lipschitzly embed. In this article the geometry of the countably branching diamond graph of depth $k$, denoted $D_k^\omega$, is studied. In this introduction only a few fundamental notions and concept from metric space and Banach space geometry are recalled and the reader is directed to the core of the article or the references \cite{Handbook}, \cite{Ostrovskiibook} for undefined notions. 

\medskip

Let $(X,d_X)$ and $(Y,d_Y)$ be two metric spaces. A map $f\colon X\to Y$ is called a bi-Lipschitz embedding if there exist $s>0$ and $D\ge 1$ such that for all $x,y\in X$,
\begin{equation}\label{E:1.1_1}
s\cdot d_X(x,y)\le d_Y(f(x),f(y))\le D\cdot s\cdot d_X(x,y).
\end{equation}
As usual $c_{Y}(X):=\inf\{D\ge 1\ |\text{ equation }\eqref{E:1.1_1}\text{ holds for some embedding }f\}$ denotes the $Y$-distortion of $X$. If there is no bi-Lipschitz embedding from $X$ into $Y$ then we set $c_{Y}(X)=\infty$.
A sequence $(X_k)_{\in\bn}$ of metric spaces is said to equi-bi-Lipschitzly embed into a metric space $Y$ if $\sup_{k\in\bn}c_Y(X_k)<\infty$. 

\medskip

The research carried out in this article is motivated by exhibiting analogies or discrepancies between certain metric characterizations of local properties of Banach spaces and their asymptotic counterparts. To make this more precise some useful terminology needs to be introduced. In general, one is seeking metric characterizations in terms of metric space preclusion in a nonlinear category. For concreteness, let us state one of the numerous metaproblems that can be considered. Say, one wants to work in the uniform category where the objects are metric spaces and the morphisms are injective uniformly continous maps. Let $P$ be a property of metric spaces or Banach spaces, and $\cC_P$ the class it canonically defines. An interesting problem, is to find a family of metric spaces $(X_i)_{i\in I}$ such that $Y\in\cC_P$ if and only if the family $(X_i)_{i\in I}$ does not equi-uniformly embed into $Y$. Arguably, the most interesting case belongs to the Lipschitz category and deals with Banach space properties. The following metaproblem is general enough to encompass all the known metric characterizations in terms of metric space preclusion.

\begin{problem}[Metric space preclusion characterizations]
Fix an ambient class of Banach spaces $\cB_{amb}$. Are there Banach space properties $P$ (or equivalently the classes they define $\cB_P$), and families of metric spaces $(X_i)_{i\in I}$ such that if $Y\in \cB_{amb}$, then 
\begin{enumerate}
\item if $Y\in \cB_P$ then $\sup_{i\in I} c_Y(X_i)=\infty$,
\item if $Y\notin \cB_P$ then $\sup_{i\in I} c_Y(X_i)<\infty$.
\end{enumerate}
Following \cite{Ostrovskii2014}, the family $(X_i)_{i\in I}$ is called in that case a family of test-spaces for $P$ (or $\cB_P$) within the class $\cB_{amb}$ in the Lipschitz category.

If moreover, 
\begin{enumerate}
\item[(2')] $\sup_{Y\notin\cB_P}\sup_{i\in I}c_Y(X_i)<\infty$,
\end{enumerate}
the family $(X_i)_{i\in I}$ shall be called a uniformly characterizing family for $P$ (or $\cB_P$) within the class $\cB_{amb}$ in the Lipschitz category.
\end{problem}

Since we will only deal with the Lipschitz category in this article we will from now on drop the mention to the Lipschitz category. We are mainly concerned with sequences of test-spaces in this article and if the ambient class is the class of all Banach spaces one shall simply say ``sequence of test-spaces'' or ``uniformly characterizing sequence''. Note that if $(X_k)_{k\in \bn}$ is a sequence of test-spaces for $\cB_P$ then $Y\in \cB_P \iff \sup_{k\in \bn} c_Y(X_k)=\infty$, or equivalently, $Y\notin \cB_P \iff \sup_{k\in \bn} c_Y(X_k)<\infty$. Remark that whenever $\cB_P$ admits a sequence of test-spaces then $\cB_P$ is stable under bi-Lipschitz embeddability, in particular $P$ must be hereditary (i.e. passes to subspaces). Also, when $\sup_{k\in\bn} c_Y(X_k)=\infty$, estimating the rate of growth of $c_Y(X_k)$ in terms of some numerical parameter that can be associated to $\cB_P$ and the sequence can be fundamental for applications.

\medskip

Ribe's rigidity theorem \cite{Ribe1976} suggests that it is reasonable to believe that local properties of Banach spaces could be characterized in purely metric terms (not necessarily in terms of test-spaces though). The first successful step in the Ribe program was obtained by Bourgain when he showed that the sequence $(B_k)_{k\in\bn}$ of binary trees of height $k$ is a uniformly characterizing sequence for super-reflexivity. We refer to \cite{BallBourbaki} and \cite{NaorRibe} for a thorough description of the Ribe program and its successful achievements. It is worth noticing at this point that an analogue of Ribe's rigidity theorem in the asymptotic setting remains elusive. Nevertheless, a viable analogue of Bourgain's super-reflexivity characterization in the asymptotic setting was obtained in \cite{BKL2010}. Before we state the Baudier-Kalton-Lancien characterization, let us mention that when dealing with asymptotic properties one shall restrict oneself to the class of reflexive Banach spaces. Whether reflexivity is a technical or a conceptual requirement seems to be a challenging and delicate issue. Also, the Baudier-Kalton-Lancien characterization could be considered as the first step in a certain asymptotic declination of the Ribe program which would seek for characterizations of asymptotic properties in purely metric terms. Important techniques, deep contributions, and profound ideas to tackle such a program are ubiquitous in recent work of the late Nigel Kalton (and his coauthors; c.f. \cite{Kalton2007}, \cite{KaltonRandrianarivony2008}, \cite{BKL2010}, \cite{KaltonFM11}, \cite{KaltonMA12}, \cite{KaltonTAMS13}, \cite{KaltonIJM13}). 

\medskip

Recall that a Banach space $Y$ is said to be uniformly convexifiable if $Y$ admits an \textit{equivalent norm} that is uniformly convex, and the convenient notation $Y\in\langle UC\rangle$ shall be used in the sequel. Similar notations and acronyms are used without further explanations for renorming with other properties. The Baudier-Kalton-Lancien characterization states that the sequence $(T_k^\omega)_{k\in\bn}$ of countably branching trees of height $k$ is a uniformly characterizing sequence for the class of Banach spaces that are asymptotically uniformly smoothable and asymptotically uniformly convexifiable, within the class of reflexive Banach spaces. The paramount ingredient to achieve this characterization is the Szlenk index. Indeed in the reflexive setting,  the Banach spaces that are asymptotically uniformly smoothable and asymptotically uniformly convexifiable are exactly those whose Szlenk index and the Szlenk index of their dual is at most $\omega$. Around the same time, Johnson and Schechtman \cite{JohnsonSchechtman2009} found two new uniformly characterizing sequences for super-reflexivity, namely the sequence $(D_k^2)_{k\in\bn}$ of (2-branching) diamond graphs and the sequence $(L_k^2)_{k\in\bn}$ of (2-branching) Laakso graphs. Very recently, Ostrovskii and Randrianantoanina \cite{ORpreprint} showed that for any $r\ge 2$ the sequence $(D_k^r)_{k\in\bn}$ of $r$-branching diamond graphs is also a uniformly characterizing sequence for super-reflexivity. In \cite{BaudierZhang2016}, Baudier and Zhang gave a different proof that $\sup_{k\in\bn} c_Y(T_k^\omega)=\infty$ where $Y$ is any reflexive Banach space that is asymptotically uniformly smoothable and asymptotically uniformly convexifiable. This new proof is based on the fact that such spaces have an equivalent norm that has property $(\beta)$ of Rolewicz (see \cite{DKLR}) and could be easily adjusted, for the same Banach space target, to show that $\sup_{k\in\bn} c_Y(L_k^\omega)=\infty$ and $\sup_{k\in\bn} c_Y(P_k^\omega)=\infty$ where $(L_k^\omega)_{k\in\bn}$ (resp. $(P_k^\omega)_{k\in\bn}$) is the sequence of countably branching Laakso (resp. parasol) graphs. Unfortunately the geometric argument used in this proof could not settle the similar problem for the sequence of countably branching diamond graphs. This issue is resolved in this article. Also, a question that arises naturally is whether one could replace the sequence $(T_k^\omega)_{k\in\bn}$ in the Baudier-Kalton-Lancien characterization by $(D_k^\omega)_{k\in\bn}$, $(L_k^\omega)_{k\in\bn}$ or $(P_k^\omega)_{k\in\bn}$ and thus obtain another analogy between the local and the asymptotic versions of the Ribe program. One of the main application of our work is that one cannot replace the sequence $(T_k^\omega)_{k\in\bn}$ neither with the sequence $(D_k^\omega)_{k\in\bn}$, $(L_k^\omega)_{k\in\bn}$ nor $(P_k^\omega)_{k\in\bn}$ in the Baudier-Kalton-Lancien characterization. This seemingly discrepancy between the local Ribe program and the asymptotic Ribe program is due to the fact that the class of super-reflexive Banach spaces, the class of uniformly convexifiable Banach spaces, and the class of uniformly smoothable Banach spaces form the same class in different disguise! However, the classes of uniform asymptotically convexifiable and uniform asymptotically smoothable spaces differ. Actually our work shows that the sequence $(D_k^\omega)_{k\in\bn}$ is a uniformly characterizing sequence for the class of asymptotically uniformly convexifiable Banach spaces within the class of separable reflexive Banach spaces with an unconditional asymptotic structure.


\subsection{Content of the article}

In Section \ref{S:2} two descriptions of the countably branching diamond graph $D_k^\omega$ are given: a recursive graph-theoretical description and a non-recursive set-theoretical description. The two constructions define the same unweighted graph (up to graph isomorphism) and therefore the same metric space (up to isometry). The basic metric properties of the diamond graphs are recorded in Section \ref{S:2.3} for convenience.

\medskip

Section \ref{S:3} is divided into four parts, three of them dealing with the embeddability of the sequence $(D_k^\omega)_{k\in\bn}$ into certain Banach spaces. In Section \ref{S:3.1} we show that for every $k\in\bn$, $c_{Y}(D_k^{\omega})\le 6$ whenever $Y$ contains arbitrarily good $\ell_\infty^{<\omega}$-trees. In Section \ref{S:3.2} a new proof is given showing that for any $k\in\bn$, $D_k^\omega$ admits a bi-Lipschitz embedding into $L_1[0,1]$ that has distortion at most $2$. Note that the embedding possesses some interesting properties, such as being ``vertically isometric''. Both embeddings requires the non-recursive set-theoretical description of the countably branching diamond graphs. In Section \ref{S:3.3} embeddability into Banach spaces of the form $L_p([0,1],Y)$ is discussed. It is shown that for $1 \le p < \infty$, if $Y$ is not super-reflexive, then for every $k\in\bn$, $c_{L_p([0,1],Y)}(D_k^{\omega})\le2^{1+1/p}$. This result has an interesting consequence in the renorming theory that is discussed in Section \ref{S:6}. The remaining part, Section \ref{S:4}, is devoted to providing sufficient conditions for a Banach space to contain arbitrarily good $\ell_\infty$-trees of arbitrary height, and thus is essentially Banach space theoretical. In particular we show that such trees can always be found in any Banach space that contains, for all $n\in\bn$, $\ell_\infty^n$ in its $n$-th asymptotic structure. Combining this result with the embedding result from Section \ref{S:3.1} one obtains our main embedding theorem.

\begin{mainthm}
If for all $n\in\bn$, $\ell_\infty^n$ is in the $n$-th asymptotic structure of $Y$, then $\sup_{k\in\bn} c_Y(D_k^\omega)<\infty$, i.e. the sequence $(D^\omega_k)_{k\in\bn}$ can be equi-bi-Lipschtizly embedded into $Y$.
\end{mainthm} 

Finally, we show that any separable reflexive Banach space with an unconditional asymptotic structure whose Szlenk index of its dual is larger than $\omega$ contains $\ell_\infty^n$ in its $n$-th asymptotic structure for all $n\in\bn$. To achieve this, some tools and concepts from asymptotic Banach space theory (e.g. trees and branches in Banach spaces, asymptotic structure...) are needed and they are recalled in Section \ref{S:4.1}. 

\medskip

In Section \ref{S:5} we are concerned with finding obstructions to the embeddability of the diamond graphs. The main result of this section states that one cannot embed equi-bi-Lipschitzly the sequence $(D_k^\omega)_{k\in\bn}$ into a Banach space that is midpoint asymptotically uniformly convex. The proof of the main result builds upon two very useful techniques: an approximate midpoint argument (originally from Enflo) and the self-improvement argument \`a la Johnson and Schechtman. The proof can be significantly generalized to handle a quite large collection of sequences of graphs, which in particular contains the countably branching Laakso graphs and the countably branching parasol graphs.

\medskip

The article ends with a section devoted to the applications of our work in metric geometry, and in renorming theory, some of which being already mentioned in the first part of this introduction. The two main applications are a new metric characterization in terms of graph preclusion, and tight estimates on the $L_p$-distortion of the countably branching diamond graphs. Our asymptotic notation conventions are the following. Throughout this article we will use the notation $\lesssim$ or $\gtrsim$, to denote the corresponding inequalities up to universal constant factors. We will also denote equivalence up to universal constant factors by $\approx$, i.e., $A\approx B$ is the same as $(A\lesssim B)\land(A\gtrsim B)$. 

\begin{mainthm} Let $Y$ be a reflexive Banach space with an unconditional asymptotic structure. Then, $Y$ is asymptotically uniformly convexifiable if and only if $\sup_{k\in\bn} c_{Y}(D_k^\omega)=\infty.$
\end{mainthm} 

 \begin{mainthm}
 \begin{enumerate}[i)]
 \item For $1 \le p < \infty$, $c_{\ell_p}(D_k^{\omega})\approx k^{1/p}$.
 \item  For $1 < p < \infty$, $c_{L_p}(D_k^{\omega})\approx \min\{k^{1/p}, \sqrt{k})\}$.
 \end{enumerate}
 \end{mainthm} 

We conclude this section by gathering all the known (as far as we know) metric space preclusion characterizations in the following three tables. Note that all the known sequence of test-spaces are actually uniformly characterizing sequences. We also included some interesting open problems. Needless to say that they are many more open problems.
\vskip 1cm 
\fontsize{7}{8}\selectfont
\begin{table}[h]
\begin{tabular}{|c|c|c|c|}
\hline
 & & & \\
$\cB_{amb}$&$\cB_P$ & test-spaces&  \\
 & & & \\
\hline
 & & & \\
 & super-reflexivity& $(B_k)_{k\in\bn}$, $(T_k^r)_{k\in\bn}$ $r\ge 3$ & Bourgain (\cite{Bourgain1986a},1986) \\
 
 & $\langle UC\rangle$ & $B_\infty$ & Baudier (\cite{Baudier2007},2007)\\

  & $\Dz(Y)=\omega$ & $(D_k^2)_{k\in\bn}$, $(L_k^2)_{k\in\bn}$ & Johnson-Schechtman (\cite{JohnsonSchechtman2009},2009)\\

all Banach spaces & $\langle US\rangle$ & certain hyperbolic groups & Ostrovskii (\cite{Ostrovskii2014AGMS},2014)\\

 & $\Dz(Y^*)=\omega$ & $(D_k^r)_{k\in\bn}$, $r\ge 3$ & Ostrovskii-Randrianantoanina \\
  & & & (\cite{ORpreprint}, 2016) \\
\cline{2-4}
 & & & \\
 & K-convexity & Hamming cubes $(H_k)_{k\in\bn}$ & Bourgain-Milman-Wolfson \\
  & non-trivial type & & (\cite{BourgainMilmanWolfson1986},1986) \\
   & & & \\
\cline{2-4}
 & & & \\
 & non-trivial cotype & $\ell_\infty$-grids $([m]^n_\infty)_{m,n\in\bn}$ & Mendel-Naor (\cite{MendelNaor2008}, 2008)\\
  & & & \\
\hline 
& & & \\ 
& UMD  & open & \\ 
& & & \\ 
\hline 
& & & \\ 
& Pisier's property $(\alpha)$ & open & \\
& & & \\
\hline
\end{tabular}
\vskip .2cm
\caption{Local Ribe program}
\end{table}

\newpage{}

\fontsize{6}{8}\selectfont
\begin{table}[h]
\begin{tabular}{|c|c|c|c|}
\hline
 & & & \\
$\cB_{amb}$&$\cB_P$ & test-spaces &  \\
 & & & \\
\hline
 & & & \\
 & $\langle AUC\rangle\cap\langle AUS\rangle$   & $(T_k^\omega)_{k\in\bn}$, $T_\infty^\omega$ & Baudier-Kalton-Lancien \\
  & & & (\cite{BKL2010},2010)\\
 & $\langle (\beta)\rangle$   &  & \\
  & & & \\
reflexive Banach spaces & $\max\{\Sz(Y),\Sz(Y^*)\}=\omega$  & $(\cS_1, d_{1,1})$ & Motakis-Schlumprecht \\
 & & & (\cite{MotakisSchlumprecht},2016)\\
  & & & \\
\cline{2-4}
 & & & \\
& $\max\{\Sz(Y),\Sz(Y^*)\}\le \omega^\alpha$ & $(\cS_\alpha, d_{1,\alpha})$ & Motakis-Schlumprecht \\
 & when $\omega^\alpha=\alpha$& & (\cite{MotakisSchlumprecht},2016)\\
  & & & \\
\hline
 & & & \\ 
reflexive Banach spaces & $\langle AUC\rangle$ &  & \\
with an unconditional    & & $(D_k^\omega)_{k\in\bn}$ & this article\\
  asymptotic structure   & $\Sz(Y^*)=\omega$ & & \\
 & & & \\
 & & & \\
\hline
 & & & \\
  & $\langle AUS\rangle$ &  & \\
  & & open  & \\
   & $\Sz(Y)=\omega$ & & \\
   & & & \\
\hline
\end{tabular}
\vskip .2cm
\caption{Asymptotic Ribe program}
\end{table}

\vskip -.5cm
\fontsize{6}{8}\selectfont
\begin{table}[h]
\begin{tabular}{|c|c|c|c|}
\hline
 & & & \\
$\cB_{amb}$&$\cB_P$ & test-spaces&  \\
 & & & \\
\hline
 & & & \\
dual Banach spaces & Radon-Nikod\'ym & $D_\infty$ & Ostrovskii (\cite{Ostrovskii2014}, 2014)\\
 & & & \\
\hline
 & & & \\
 & reflexivity & open  & \\
 & & & \\
\hline
\end{tabular}
\vskip .2cm
\caption{Extended Ribe program}
\end{table}

\fontsize{11}{13}\selectfont

\section{The countably branching diamond graphs}\label{S:2}
By a graph $G=(V,E)$ we mean a graph with vertex-set $V$ and edge-set $E$. The set of vertices or edges is allowed to be infinite however we restrict our attention to simple graphs. i.e. without multiple edges. Our graphs will always be unweighted and equipped with the shortest path distance. The geometry of the (binary or $2$-branching) diamond graphs has proved to be fundamental in connection with applications in theoretical computer science (for the dimension reduction problem in $\ell_1$ see for instance \cite{LeeNaor2004}). The diamond graph of depth $k$ is build in a recursive fashion starting with the $4$-cycle (the diamond graph of depth $1$). The diamond graph of depth $2$ is obtained by replacing each edge of the diamond graph of depth $1$ with a copy of the $4$-cycle. The diamond graphs of higher depth are build in similar fashion. The sequence of graphs obtained, denoted $(D_k^2)_{k\in\bn}$, is usually simply refered to as the sequence of diamond graphs\footnote{Note that in some articles the recursive construction is slightly different since the diamond graph of depth $k+1$ is obtained by replacing each edge of the diamond graph of depth $k$ by a copy of itself; this actually gives a subsequence of the sequence of the diamond graphs as defined here.}. In this article we are concerned with the sequence of \textit{countably branching} diamond graphs, denoted $(D_k^\omega)_{k\in\bn}$, whose formal definition is presented in the next section.
 
\subsection{Graph theoretical recursive definition}

A directed $s$-$t$ graph $G=(V,E)$ is a directed graph which has two distinguished vertices $s,t\in V$. To avoid confusion, we will also write sometimes $s(G)$ and $t(G)$. There is a natural way to ``compose'' directed $s$-$t$ graphs using the $\oslash$-product defined in \cite{LeeRaghavendra2010}. Given two directed $s$-$t$ graphs $H$ and $G$, define a new graph $H\oslash G$ as follows:
\begin{enumerate}[i)]
\item $V(H\oslash G):=V(H)\cup(E(H)\times(V(G)\backslash\{s(G),t(G)\}))$
\item For every oriented edge $e=(u,v)\in E(H)$, there are $|E(G)|$ oriented edges,
\begin{align*}
&\big\{(\{e,v_1\},\{e,v_2\})| (v_1,v_2)\in E(G) \text{ and }v_1,v_2\notin \{s(G),t(G)\}\big\}\\
\cup &\big\{(u,\{e,w\}) | (s(G),w)\in E(G)\big\}\cup \big\{(\{e,w\},u) | (w, s(G))\in E(G)\big\}\\
\cup& \big\{(\{e,w\},v) | (w,t(G))\in E(G)\big\}\cup \big\{(v,\{e,w\}) | (t(G),w)\in E(G)\big\}
\end{align*}
\item $s(H\oslash G)=s(H)$ and $t(H\oslash G)=t(H)$.
\end{enumerate}

It is also clear that the $\oslash$-product is associative (in the sense of graph-isomorphism or metric space isometry), and for a directed graph $G$ one can recursively define $G^{\oslash^{k}}$ for all $k\in\bn$ as follows:
 \begin{itemize}
 \item $G^{\oslash^{1}}:= G$.
 \item $G^{\oslash^{k+1}}:=G^{\oslash^{k}}\oslash G$, for $k\ge 1$.  
\end{itemize}

Note that it is sometimes convenient, for some induction purposes, to define $G^{\oslash^{0}}$ to be the two-vertex graph with an edge connecting them. Note also that if the base graph $G$ is symmetric the graph $G^{\oslash^{k}}$ does not depend on the orientation of the edges.

\medskip

Consider the complete bipartite infinite graph $K_{2,\omega}$ with two vertices on one side, (such that one is $s(K_{2,\omega})$ and the other $t(K_{2,\omega})$), and countably many vertices on the other side. The countably branching diamond graph of depth $k$ is defined as $D_k^\omega:=K_{2,\omega}^{\oslash^{k}}$. If one starts with the complete bipartite graph $K_{2,r}$ for some $r\ge 2$ instead, the graph obtained is the $r$-branching diamond graph of depth $k$. In particular $D_k^2:=K_{2,2}^{\oslash^{k}}$.

\medskip

The recursive definition of the various types of diamond graphs (and the basic metric properties that can be derived from it) is usually sufficient to prove metric statement about them. However in this article we will need a more ``concrete'' representation of the countably branching diamond graphs in order to prove our main embedding result. This representation, or coding, is the purpose of the next section.
\subsection{Non recursive definition}
We denote by $[\bn]$ all subsets of $\bn$, by $[\bn]^{<\omega}$ all finite subsets,  and by $[\bn]^{\omega} $ all infinite subsets. For $k\in \bn_0:=\bn\cup\{0\}$
we denote by $[\bn]^{\le k}$ the subsets of $\bn$ which have at most $k$ elements, by  
$[\bn]^{<k}$ the subsets of $\bn$ which have less than $k$ elements, and by  $[\bn]^k$, the subsets of $\bn$,  which have exactly $k$ elements.
The elements of subsets of $\bn$ are always written in increasing order, so writing  $A=\{a_1,a_2,\ldots, a_n\}\in [\bn]^{<\omega}$, or $B=\{b_1,b_2,\ldots \}\in [\bn]^{\omega}$, means that $a_1<a_2<\ldots< a_n$, and $b_1<b_2<b_3<\dots$. For $A=\{a_1,\ldots,a_m\}\kin [\bn]^{<\omega}$,  and  $B=\{b_1,b_2,\ldots , b_n\}\kin[\bn]^{<\omega}$ or $B=\{b_1,b_2,\ldots \}\in [\bn]^{\omega}$, $A$ is called an {\em initial segment} of $B$, or $B$ is called an {\em extension of}  $A$, if 
$n>m$ (or $B$ is infinite), and  $a_1=b_1$, $a_2=b_2$,$\ldots$, $a_m=b_m$, and we write in that case $A\prec B$.
For $A=\{a_1,a_2,\ldots ,a_n\}\in [\bn]^{<\omega}$ and $m\le n$ we put $A|_m=\{a_1,a_2,\ldots,a_m\}$. By convention, $A|_0=\emptyset$. Similarly we define $A|_m$ if $A\in[\bn]^{\omega}$.
For $k\in\bn_0$ we define $\bb_0=\{0,1\}$ and  $\bb_k:=\Big\{ \sum_{i=1}^{k-1} \sigma_i 2^{-i}+2^{-k}: (\sigma_i)_{i=1}^{k-1}\subset\{0,1\}\Big\}.$
It is easy to see that $\bb_k=\{\frac{1}{2^k}, \frac{3}{2^k}, \dots, \frac{2i-1}{2^k},\dots, \frac{2^k-1}{2^k}\}$ and $|\bb_k|=2^{k-1}$ for $k\ge 1$.

\begin{defn}[Non recursive definition of the diamond graphs]
For $k\in\bn_0$ we define the graph $G_k=(V_k, E_k)$ where the set of vertices is 
\begin{align*} 
 V_k&=\big\{ (A,r): A\in[\bn]^{\le k}, \, r\in\bb_{|A|}\big\}
 \intertext{and the set of edges is}
 E_k&=\big\{ \{(A,r),(B,s)\}\colon A\prec B \text{ and } |r-s|=2^{-k}\big\}. 
 \end{align*}
\end{defn}

Let us explicitly write down the first two graphs. 
\begin{align*}
&V_0=\big\{ (\emptyset, 0), (\emptyset, 1)\big\},\, E_0=\big\{ \big\{ (\emptyset, 0), (\emptyset, 1)\big\}\big\} 
\intertext{ and putting  $b^\omega\keq(\emptyset, 0)$ and $t^\omega\keq (\emptyset, 1)$}
&V_1=\{b^\omega,t^\omega\}\cup\big\{ \big(\{j\},\sh\big) :j\in \bn\big\} \text{ and }E_1=\big\{ \{b^\omega,(\{j\},\sh)\},  \{t^\omega,(\{j\}, \sh)\}   :j\in\bn\big\}.  
\end{align*}

\psset{xunit=1cm, yunit=.85cm} 
\begin{pspicture}
\rput(0,-4){

\uput[0](0,12){$D_0^\omega$}
\psdots(1,13)   \uput[0](1,13){$(\emptyset,1)$}
\psdots(1,11)  \uput[0](1,11){$(\emptyset,0)$}
  															
\psline(1,13)(1,11)

\uput[0](4,12){$D_1^\omega$}
										                        \psdots(8,13)
										                       \uput[0](8,13.1){\tiny{$(\emptyset,1)$}}
\psdots(6,12)   \psdots(6.5,12)\psdots(7,12)\psdots(7.5,12)                \psdots(8,12)                     \psdots(8.5,12)     \psdots(9,12)    \psdots(9.5,12) \psdots(10,12) 
\uput[0](5.2,12.3){\tiny{$(\{1\},\sh)$}}					              \uput[0](7.85,12.2){\tiny{$(\{j\},\sh)$}}
 												\psdots(8,11)
										                  \uput[0](8,10.8){\tiny{$(\emptyset,0)$}}

\psline(6,12)(8,13)      \psline(6.5,12)(8,13)     \psline(7,12)(8,13)    \psline(7.5,12)(8,13) \psline(8,12)(8,13)    \psline(9.5,12)(8,13)   \psline(10,12)(8,13)    
\psline(6,12)(8,11)    \psline(6.5,12)(8,11)     \psline(7,12)(8,11)  \psline(7.5,12)(8,11) \psline(8,12)(8,11)   \psline(9.5,12)(8,11)   \psline(10,12)(8,11)

\uput[0](0,6){$D_2^\omega$}
\pspolygon(6,9)(8.5,7.5)(10,5)(8.5,2.5)(6,1)(3.5,2.5)(2,5)(3.5,7.5)(6,9)
\pspolygon(6,9)(7.5,6.5)(10,5)(7.5,3.5)(6,1)(4.5,3.5)(2,5)(4.5,6.5)(6,9)
\pspolygon(6,9)(6.9,6.5)(6,5)(6.9,3.5)(6,1)(5.1,3.5)(6,5)(5.1,6.5)(6,9)

\psdots(6,9)(8.5,7.5)(10,5)(8.5,2.5)(6,1)(3.5,2.5)(2,5)(3.5,7.5)(6,9)
\psdots(6,9)(7.5,6.5)(10,5)(7.5,3.5)(6,1)(4.5,3.5)(2,5)(4.5,6.5)(6,9)
\psdots(6,9)(6.9,6.5)(6,5)(6.9,3.5)(6,1)(5.1,3.5)(6,5)(5.1,6.5)(6,9)

\psdots(3,5)(4,5)(5,5)(7,5)(8,5)(9,5)
\psdots(4,7)(8,7)(8,3)(4,3)
\psdots(3.75,7.25)(4.25,6.75)(7.75,6.75)(8.25,7.25)
\psdots(3.75,2.75)(4.25,3.25)(7.75,3.25)(8.25,2.75)

\psdots(5.5,3.5)(6,3.5)(6.5,3.5)(5.5,6.5)(6,6.5)(6.5,6.5)

\uput[0](6,9){\tiny{$(\emptyset,1)$}}
\uput[0](2,7.5){\tiny{$(\{1,2\},\frac34)$}}
\uput[0](1,5.5){\tiny{$(\{1\},\sh)$}}\uput[0](6.3,5.3){\tiny{$(\{j\},\sh)$}}
\uput[0](2,2.5){\tiny{$(\{1,2\},\frac14)$}}
\uput[0](6,1){\tiny{$(\emptyset,0)$}}

}
\end{pspicture}
\vskip 4cm
In the following lemma we gather the basic combinatorial properties of the sequence $(G_k)_{k\in\bn}$ that will be needed in the sequel. Elements of proofs are only given for the facts that are not completely obvious.

\begin{lemma}\label{L:2.1}
\begin{enumerate}[a)]
\item For $j,k\in\bn$, with $j\le k$ it follows that $V_j\subset V_k$.
\item Assume that $k\in\bn$ and $(A,r)$, $(B,s)$ are in $V_k$ with $\big\{(A,r),(B,s)\big\}\in E_k$.
 Then it  follows that $|r-s|=2^{-k}$  and 
 either $|A|=k$ and $|B|<k$ or $|B|=k$ and $|A|<k$.

\item Let $k\in\bn$. If $(B,s)\in V_k$, with $B\in[\bn]^k$ (and thus $s\in \bb_k$), and if $(A,r)\in V_k$, then it follows that $\{ (A,r), (B,s)\}\in E_k$ if and only if 

\begin{equation}\label{E:2.2_1}
\hskip -9cm A\prec B
\end{equation}
and
\begin{equation}\label{E:2.2_2} 
\hskip -3cm \text{either $r=s-2^{-k}=\sum_{i=1}^{i^-} \sigma_i 2^{-i}$, and $A=B|_{i^-}$,}
\end{equation}

\begin{align*}
&\text{if $i^-=\max\{ 1\le i\le k-1\colon \sigma_i=1\}$ exists, and $A=\emptyset$ and $r=0$ otherwise},\\
&\text{ or $r=s+2^{-k}=\sum_{i=1}^{i^+-1} \sigma_i 2^{-i} + 2^{-i^+}$, and $A=B|_{i^+}$}\\
&\text{if $i^+=\max\{ 1\le i\le k-1 : \sigma_{i}=0\}$ exists, and $A=\emptyset$ and $r=1$ otherwise}.
\end{align*}
\item For a given $(B,r)\in [\bn]^k\!\times \!\bb_k$  there is a unique edge $\big \{(B_+,r+2^{-k}),(B_-,r-2^{-k})\}\in E_{k-1}$,
for which we have  $$\{(B,r),(B_+,r+2^{-k})\}\in E_k \text{ and } \{(B,r),(B_-,r-2^{-k})\}\in E_k.$$
\end{enumerate}
\end{lemma}

\begin{proof}
\begin{enumerate}
\item[c)] Assume $k\ge 1$. From the  definition  of $E_k$ it follows that if for two elements $u$ and $v$ in $V_k$, say $u= (A,r)$ and $v=(B,s)$,  we have $\{u,v\}\in E_k $, then  either $A$ or $B$ has cardinality $k$ and the other set does have cardinality less than $k$.
So assume that $|A|<|B|=k$. Then write $s$ as $s=\sum_{i=1}^k \sigma_i 2^{-i}$, with $\sigma_i\in\{0,1\}^k$, for $i\in\{1,2,\dots,k\}$ and $\sigma_k=1$. Now since $|r-s|=2^{-k}$ and $r\in \bb_{|A|}$ it follows that,

$$r=\begin{cases}
        \sum_{i=1}^{m} \sigma_i 2^{-i} &\text{ if $r<s$ }\\
        \sum_{i=1}^{m-1} \sigma_i 2^{-i} +2^{-m}&\text{ if $r>s$.}
        \end{cases}$$
where $m:=|A|$ and
$$|A|=\begin{cases}
             \max\{ 1\kleq  i\kleq k-1 , \sigma_i\keq1\}, &\text{if $r\kle s$}\\
             \max\{ 1\le i\le k-1  : \sigma_{i}=0\}, &\text{if $r>s$,}
             \end{cases}$$
with $\max(\emptyset):=0$.
\item[d)] Indeed, write $r$ as 
$$r=\sum_{i=1}^k \sigma_i 2^{-i},\text{ with $\sigma_i\in \{0,1\}^k$, and $\sigma_k=1$}.$$
 
Put  $$i^-=\max\{ 1\le i\le k-1 :\sigma_i=1\},$$ and $$i^+=\max\{1\le i  \le k-1:\sigma_{i}=0\},$$ with $\max(\emptyset):=0$.
Then letting $B_-=B|_{i^-}$ and $B_+=B|_{i^+}$,  we deduce that  $$\big\{ (B_-,r-2^{-k}) , (B_+,r+2^{-k})\big\}\in E_{k-1},$$
and $$\big\{(B,r), (B_-,r-2^{-k})\big\} ,\big\{ (B,r), (B_+,r+2^{-k})\big\}\in E_{k}.$$ 
The uniqueness is then clear since according to equation \eqref{E:2.2_1} it is necessary that $A\prec B$ for $\{(A,s),(B,r)\}$ to be in $E_k$, while equation \eqref{E:2.2_2} allows only two possible cardinalities for $A$.
\end{enumerate}
\end{proof}

The graph $G_k$ introduced above is nothing else but a concrete representation of the countably branching diamond graph of depth $k$ defined recursively in the previous section.

\begin{proposition}
For all $k\ge 0$, $G_k$ is graph isomorphic to $D_k^\omega$.
\end{proposition} 

\begin{proof}
Recall that $D_k^\omega:=(V_k^\omega, E_k^\omega)$ and $G_k:=(V_k,E_k)$. We will think of $s(D_k^\omega)$ (resp. $t(K_{2,\omega})$) to be the bottom vertex (resp. the top vertex) of $D_k^\omega$, we shall use the notation $b_k^\omega$ for $s(D_k^\omega)$ and $t_k^\omega$ for $t(D_k^\omega)$.
We shall prove by induction on $k\in\bn_0$ the following statement:

\medskip
 
$H_k$: there exists a graph isomorphism $\varphi_{k}:D_k^\omega\to G_k$ such that 
\begin{enumerate}[i)]
\item for all $n\le k-1$, $\varphi_k|_{D_n^\omega}=\varphi_n$,
\item and if $e^\omega=\{u,v\}\in E_k^\omega$, and if $(A,r)=\varphi_k(u)$ and $(B,s)=\varphi_k(v)$ then $r=s+2^{-k}$.
\end{enumerate}

For the base case, define the map $\varphi_0: V_0^\omega\to V_0$, $t_0^\omega\mapsto (\emptyset,1)$, and $b_0^\omega\mapsto (\emptyset,0)$. It is clear that $\varphi_0$ statisfies $H_0$.

\medskip

Now assume that $(H_k)$ holds and recall that $V_{k+1}^\omega=V_k^\omega\cup \big\{ (e^\omega,j) : e^\omega\in E_k^\omega,\, j\in \bn \big\}$. We define, as required for the old vertices, $\varphi_{k+1}(v)=\varphi_k(v)$ if $v\in V_k^\omega$. For the new vertices $(e^\omega,j)$ where $e^\omega=\{u,v\}\in E_k^\omega$ and $j\in \bn$ we choose $\varphi_{k+1}((e^\omega,j))$ as follows. If $\varphi_k(u)=(A,r)$ and $\varphi_k(v)=(B,s)$, then by the induction hypothesis $\big\{(A,r), (B,s)\big\}\in E_k$, and we can, using Lemma \ref{L:2.1} (c), assume without loss of generality 
that $B=\{b_1,b_2,\ldots, b_k\}\in[\bn]^k$, with $b_1<b_2<\ldots <b_k$, $s\in \bb_k$, $A\prec B $ and $r=s\pm 2^{-k}\in \bb_m$, with $m<k$. So we put
$$\varphi_{k+1}((e^\omega,j)):= \Big( B\cup\{b_{k}+j\}, \frac{r+s}2\Big).$$
First we note that since $s\in \bb_k$ and $r\in \bb_m$,  for some $m<k$, it follows that $\frac{r+s}2\in \bb_{k+1}$. Thus $\varphi_{k+1}((e^\omega,j))\in V_{k+1}\setminus V_k$.
This shows that $\varphi_{k+1}$ is well defined, and that $\varphi_{k+1}(V_{k+1}^\omega\setminus V_k^\omega)\subset V_{k+1}\setminus V_k$.
The claim that $\varphi_{k+1}$ is bijective and that for $u,v\in V_k^\omega$, $\{u,v\}\in E_k^\omega$ $\iff$ $\{\varphi_{k+1}(u),\varphi_{k+1}(v)\}\in E_n$, can be now obtained 
from Lemma \ref{L:2.1} (d) above, the definition of $D_{k+1}^\omega$ and the induction hypothesis.

\end{proof}

\subsection{Basic metric properties of the countably branching diamond graphs}\label{S:2.3}

In this section we discuss the basic metric properties of the countably branching diamond graphs which will be needed in the sequel. 

First of all, let $d_k$ denote the shortest path metric on $D_k^\omega$. Since for two vertices $(A,r)$ and $(B,s)$ in $V_k$ to form an edge in $D_k^\omega$ it is necessary that $|r-s|=2^{-k}$, it follows that 
\begin{equation}
\label{E:2.3_1} 
d_k\big((A,r),(B,s)\big)\ge |r-s|2^{k}.
\end{equation}

Secondly we observe that for any $(A,r)\in V_k$ 
\begin{equation}
\label{E:2.3_2} 
d_k\big(b_k^\omega, (A,r)\big)=r2^{k}\text{ and }d_k\big((A,r),t_k^\omega\big)=(1-r)2^{k}.
\end{equation}
Equality \eqref{E:2.3_2} can be generalized as follows. We say that two vertices $(A,r)$ and $(B,s)$ in $D_k^\omega$ lie on the same vertical path if there exists a simple increasing path $\cP=\big((P_0,p_0),(P_1,p_1),\ldots, (P_n,p_n)\big)$ of length $n$ for some $n\le 2^k$ in $D_k^\omega$, \ie $\big\{(P_{m-1},p_{m-1}),(P_{m},p_{m})\big\}\in E_k$ for $m=1,2,\ldots, n$ and $p_0<p_1<\ldots< p_n$ (and thus $p_m=m2^{-k}+p_0$), such that $(A,r)=(P_0,p_0)$ and $(B,s)=(P_n,p_n)$ or $(B,s)=(P_0,p_0)$ and $(A,r)=(P_n,p_n)$. In that case we  observe that  $\cP$ is the shortest path connecting $(P_0,p_0)$ and $(P_n,p_n)$ and thus 
\begin{equation}\label{E:2.3_3} 
d_k\big((A,r), (B,s)\big)=d_k\big((P_n,p_n), (P_0,p_0)\big)=2^{k}(p_n-p_0)=2^{k}|r-s|.
\end{equation}

Note that for $k\in\bn$ we can write $V_{k}$ as $V_{k}=\bigcup_{j=1}^\infty V^{(j,+)}_{k} \cup \bigcup_{j=1}^\infty V^{(j,-)}_{k}$ where for $j\in\bn$,
\begin{align*}
V^{(j,+)}_{k}= &\Big\{ \Big(\{j\}\cup A, {\scriptstyle\frac{r+1}{2}}\Big):  (A,r)\kin V_{k-1},\,  0< r<1, \text{ and }j<\min(A)\Big\}\\ 
&\bigcup \big\{t_{k}^\omega,(\{j\},\sh)\big\} \\
V^{(j,-)}_{k}= &\Big\{ \Big(\{j\}\cup A, \frac{r}{2}\Big):  (A,r)\kin V_{k-1},\,  0< r< 1, \text{ and }j<\min(A)\Big\}\\
&\bigcup\big\{b_{k}^\omega,(\{j\},\sh)\}.
\end{align*}

\medskip

Let  $k\in\bn$ and  $i\neq j\in \bn$. Note that 
$$V^{(i,+)}_{k}\cap V^{(j,+)}_{k}=\{ t_{k}^\omega\} ,\, V^{(i,-)}_{k}\cap V^{(j,-)}_{k}=\{ b_{k}^\omega\},$$
and
$$V^{(j,+)}_{k}\cap V^{(j,-)}_{k}=\{(j,\sh)\}, V^{(i,+)}_{k}\cap V^{(j,-)}_{k}=\emptyset.$$
Secondly,  there is no edge between any element of $V^{(i,+)}_{k}\setminus\{ t_{k}^\omega\} $ and any element of $V^{(j,+)}_{k}\setminus\{ t_{k}^\omega
\} $,
 any element of $V^{(i,-)}_{k}\setminus\{ b_{k}^\omega\} $ and any element of $V^{(j,-)}_{k}\setminus\{ b_{k}^\omega\} $, any element of $V^{(i,+)}_{k}$ and any element of $V^{(j,-)}_{k} $, any element of $V^{(i,+)}_{k}\setminus\{(i,\sh)\} $ and any element of $V^{(i,-)}_{k}\setminus\{(i,\sh)\}$. It follows therefore from \eqref{E:2.3_3} that,
\begin{itemize}
\item if $i\neq j\in\bn$ then for all $(A,r)\in V^{(i,+)}_k$ and $(B,s)\in V^{(j,+)}_k$,
\begin{align}
\label{E:2.3_4}
d_{k}\big((A,r),(B,s)\big)&= d_{k}\big((A,r),t_{k}^\omega\big)+d_{k}\big(t_{k}^\omega,(B,s)\big)= (2-r-s)2^{k}
\end{align}
\item if $i\neq j\in\bn$ then for all $(A,r)\in V^{(i,-)}_k$ and $(B,s)\in V^{(j,-)}_k$,
\begin{align}
\label{E:2.3_5}
d_{k}((A,r),(B,s))&= d_{k}\big((A,r),b_{k}^\omega\big)+d_{k}\big(b_{k}^\omega,(B,s)\big)= (r+s)2^{k}
\end{align}
\item if $i\neq j\in\bn$ then for all $(A,r)\in V^{(i,+)}_k$ and $(B,s)\in V^{(j,-)}_k$,
\begin{align}
\label{E:2.3_6}
d_{k}\big((A,r),(B,s)\big)&=\\  
\min\!\Big\{ d_{k}\big((A,r),&t_{k}^\omega\big)\kplus d_{k}\big(t_{k}^\omega, (B,s)\big),d_{k}\big((A,r), b_{k}^\omega\big)\kplus d_{k}\big(b_{k}^\omega ,(B,s)\big)\Big\}\notag
\end{align}
\item if $j\in\bn$ then for all $(A,r)\in V^{(j,+)}_k$ and $(B,s)\in V^{(j,-)}_k$,
\begin{align}
\label{E:2.3_7}
d_{k}\big((A,r),(B,s)\big)&= d_{k}\big((A,r),(\{j\},{\scriptstyle \frac12})\big)+ d_{k}\big((\{j\},{\scriptstyle \frac12}),(B,s)\big)\\
&= (r-{\scriptstyle \frac12})2^{k}+ ({\scriptstyle \frac12}- s)2^{k}=(r-s)2^{k}\notag\end{align}
We can therefore deduce from \eqref{E:2.3_4}-\eqref{E:2.3_7} that if $i\neq j\in\bn$ then for $(A,r)\in V^{(i,+)}_{k}$ and $(B,s)\in V^{(j,-)}_{k}$,
\begin{align}\label{E:2.3_8}
d_{k}\big((A,r),(B,s)\big)&=2^{k}\cdot\min\big(1-r+{\scriptstyle\frac12}+ {\scriptstyle\frac12} -s,   r+s \big) \\
&=2^{k}\cdot\min\big(2-r -s,   r+s \big) \notag\\
&=2^{k} \begin{cases}   2-r-s &\text{if $r+s\ge 1$} \\
                                                 r+s   &\text{if $r+s\le 1$.}  \end{cases} 
\notag
\end{align}
\end{itemize}

For $j\in\bn$, let $A+j$ (resp. $A-j$) be the set obtained by adding (resp. substracting) $j$ to each element of $A$, with the convention that $\emptyset\pm j=\emptyset$. Define also $s_j(A):=\{j\}\cup (A+j)$ whenever $A\ne \emptyset$. Note that if $A$ is an element in $[\bn]^{\le k}$ for some $k$, then $s_j(A)$ belongs to $[\bn]^{\le k+1}$. For $j\in\bn$, if $A=\{j,a_2,\dots, a_m\}$ we also define $s_j^{-1}(A):=A\setminus\{j\}-j$ and note that $s_j^{-1}\circ s_j(A)=A$. Using the two combinatorial shifts $s_j$ and $s_j^{-1}$ one can define (essentially) two natural isometries based on the self-similarities of the diamond graphs. The first two isometries are 
\begin{align}\label{E:2.3_9}
I^{(j,-)}_k: &V^{(j,-)}_{k}\to V_{k-1}\\
                &(A,r)\mapsto \big(s_j^{-1}(A), 2r\big),\text{ if } (A,r)\!\not=\!b_{k}^\omega \text{ and } b_{k}^\omega\mapsto  b_{k-1}^\omega,\notag
\end{align}
and
\begin{align}\label{E:2.3_10}
I^{(j,+)}_k: &V^{(j,+)}_{k}\to V_{k-1}\\
                 &(A,r)\mapsto \big(s_j^{-1}(A), 2r-1)\big),\text{ if } (A,r)\!\not=\!t_{k}^\omega \text{ and } t_{k}^\omega\mapsto  t_{k-1}^\omega,\notag
\end{align}
which are the canonical isometries from a lower (resp. upper) diamond in $D_{k}^\omega$ of depth $k-1$ onto $D_{k-1}^\omega$.

The last isometry,
\begin{align}\label{E:2.3_11}
F^{(j,+)}_k:&V^{(j,+)}_{k}\to V_{k-1}\\
                 &(A,r)\mapsto \big(s_j^{-1}(A), 2(1-r)\big),\text{ if } (A,r)\!\not=\!t_{k}^\omega  \text{ and }
 t_{k}^\omega\mapsto  b_{k-1}^\omega,\notag
\end{align}
is an isometry of an upper diamond in $D_{k}^\omega$ of depth $k-1$ onto $D_{k-1}^\omega$ that flips the vertices upside down.


\section{Embeddability of the countably branching diamond graphs}\label{S:3}

\subsection{Embeddability into Banach spaces containing particular $\ell_\infty$-trees}\label{S:3.1}
In this section our main embedding theorem is proven. The embedding is based on the existence of certain trees in the target space. For the sake of clarity the study of the existence in Banach spaces of this technical device is postponed to Section \ref{S:4}.  
\subsubsection{Good $\ell_\infty$-trees of arbitrary height}\label{S:3.1.1}
A linear ordering $(A_i)_{i\in\bn}$ of the set 
 $[\bn]^{\le n}$ is called {\em compatible on $ [\bn]^{\le n}$} if 
 for any $i,j$ in $\bn$:
 \begin{align}\label{E:3.1.1_1}
 \text{if $\max(A_i)<\max(A_j)$ then $i<j$}.
 \end{align}
In other words, one can only  assign  the number $j\in \bn$ to an element  $A\in[\bn]^{\le n}$, if for all elements  $B\in [\bn]^{\le n}$, with $\max(B)<\max(A)$, were already  counted.
It is easy to see that such a linear ordering exists and that always $A_1=\emptyset$, and  $A_2=\{1\}$.

\begin{defn}
We say that a Banach space $X$ contains $(C,D)$-good $\ell_\infty$-trees of arbitrary height if there are constants $C,D\ge 1$ such that for any $n\in \bn$, 
and any compatible linear ordering $(A_i)_{i\in \bn}$  of $[\bn]^{\le n}$ there are a vector-tree $(x_A)_{A\in[\bn]^{\le n}}$ in $S_X$  and a functional-tree $(x^*_A)_{A\in[\bn]^{\le n}}$ in $S_{X^*}$
satisfying the following properties:
\begin{equation}\label{E:3.1.1_2}
\hskip -3cm\text{for all } A,B\in [\bn]^{\le n}, \text{ with }\max(B)>\max(A),
\end{equation}
\begin{equation*}
x^*_A(x_A)=1 \text{ and } x^*_A(x_B)=0,
\end{equation*}
\begin{equation}\label{E:3.1.1_3}
\hskip -1.7cm\text{ for all } (a_i)_{i=0}^n\subset \br, \text{ and all }B=\{b_1,b_2,\dots,b_n\}\in [\bn]^n,
\end{equation}
\begin{equation*}
\frac{1}{C}\|(a_i)_{i=0}^n\|_\infty\le\|a_0x_\emptyset +a_1 x_{\{b_1\}}+\cdots +a_n x_{\{b_1,\dots, b_n\}}\|_X \le C\|(a_i)_{i=0}^n\|_\infty,
\end{equation*}
\begin{equation}\label{E:3.1.1_4}
\hskip -3.3cm \text{ for every } i\le j, \text{ every } (a_m)_{m=0}^j\subset \br, \text{ one has
}
\end{equation}
\begin{equation*}
\Big\|\sum_{m=1}^i a_m x_{A_m}\Big\|_X\le D\Big\|\sum_{m=1}^j a_m x_{A_m}\Big\|_X.
\end{equation*}

If a Banach space $X$ contains for every $\vp>0$, $(1+\vp,1+\vp)$-good $\ell_\infty$-trees of arbitrary height, we say that $X$ contains good $\ell_\infty$-trees of arbitrary height almost isometrically.
\end{defn}

Note that condition \eqref{E:3.1.1_3} means that every branch $(x_A)_{A\preceq B}$ is $C^2$-equivalent to the $\ell_\infty^{n+1}$-unit vector basis, while condition \eqref{E:3.1.1_4} says that the sequence $(x_{A_i})_{i\in\bn}$ (where $(A_i)_{i\in\bn}$ is the above chosen compatible linear ordering of $[\bn]^{\le n}$) is basic, with a basis constant not exceeding $D$.
\begin{rem}
If $X$ has a bimonotone FDD  (or more generally if $X$ embeds into a space with a
bimonotone FDD), then the  sequence $(x_{A_i})_{i\in\bn}$ can chosen to be block sequence of that FDD (or an arbitrary small perturbation of a block sequence),
which implies that $(x_{A_i})_{i\in\bn}$ is also bimonotone (or has a bimonotonicity constant which is arbitrarily close to 1). 
\end{rem}

\begin{ex}\label{Ex:1}
 If $X=c_0$, then the vector-tree $(x_A)_{A\in [\bn]^{\le n}}$ together with the functional-tree $(x_A^*)_{A\in[\bn]^{\le n}}$ where
$x_A=e_{\max(A)}$ and $x^*_A=e^*_{\max(A)}$ form a $(1,1)$-good $\ell_\infty$-tree of height $n$.
\end{ex}

\subsubsection{The embedding}\label{S:3.3.2}

\begin{theorem}\label{T:3.3}
Assume $Y$ contains good $\ell_\infty$-trees of arbitrary height almost isometrically, then for every $\vp>0$ and every $k\kin\bn$ there exists $\Psi_k: D_k^\omega\to Y$, such that
if $x$ and $y$ belong to the same vertical path then
\begin{equation*}
d_k(x,y) \le \big\| \Psi_k(x)-\Psi_k(y)\big\|\le (1+\varepsilon)  d_k(x,y),
\end{equation*}
and if $x$ and $y$ do not belong to the same vertical path then 
\begin{equation*}
\frac{d_k(x,y)}{C(\vp)} \le \big\| \Psi_k(x)-\Psi_k(y)\big\|\le (1+\varepsilon)  d_k(x,y),
\end{equation*} 
where $C(\vp)=6(1+\vp)$. 

\smallskip

Moreover, if $Y$ contains good $\ell_\infty$-trees of arbitrary height almost isometrically and if $(y_{A_i})_{i\in\bn}$ is bimonotone, where $(y_A)_{A\in [\bn]^{\le n}}$ is the vector-tree, and $(A_i)_{i\in\bn}$ is a compatible linear ordering, then $\Psi_k$ can be defined such that $C(\vp)=3$.
\end{theorem}

\begin{proof}
\textit{Definition of the coefficients:} 
For $k\in\bn$, we define inductively a family of coefficients 
$\big\{c_k(i,r):\!0\kleq i\kleq k, r\kin\bigcup_{i\le m\le k}\bb_m\big\}\subset \{0,1,2\ldots 2^{k}\}$, as follows:

for $k=1$, let $c_1(0,1)=2$, $c_1(0,0)=0$, and $c_1(0,\sh)=c_1(1,\sh)=1$,

and for $k\ge 1$, $c_{k+1}(i,r)$, $0\le i\le k+1$ and $r\in \bigcup_{m=i}^{k+1} \bb_m$ will  be chosen as follows:
 \begin{align*}
 c_{k+1}(0,r)&=r2^{k+1}  \text{ for all }r\in\bigcup_{m=0}^{k+1}\bb_m  \\
 \intertext{and for $i=1,2\ldots, k+1$ and $r\in\bigcup_{m=i}^{k+1}\bb_m$, we put}
 c_{k+1}(i,r)&=\begin{cases} c_{k}(i-1,2r) &\text{if $0<r\le \sh$}\\
  c_{k}(i-1,2(1-r)) &\text{if $ \sh\le r<1$.}\end{cases}
 \end{align*}

Note that $c_{k+1}(i,r)$ is well defined for all $i=0,1,\ldots, k+1$, and $r\in \bigcup_{m=i}^{k+1} \bb_j$, since if $r=\sh$, both formulae lead to the same term, and whenever $r\in \bb_m$, for some $i\le m\le k+1$, it follows that $2r\in \bb_{m-1}$, in case that $r\le \sh$, and $2-2r\in \bb_{m-1}$ in case that $r\ge \sh$. 

\medskip

Assume that $Y$ contains good $\ell_\infty$-trees of arbitrary height almost isometrically. We shall prove by induction on $k\in\bn$ the following claim. The theorem follows easily. 

\begin{claim}\label{C:3.1} For any $\vp>0$ there is $\eta(\vp)\in(0,\vp]$ such that for every $(1+\eta(\vp),1+\eta(\vp))$-good $\ell_\infty^k$-tree whose tree or vectors (resp. tree of functionnals) is $(y_\sigma)_{\sigma\in [\bn]^{\le k}}$ (resp. $(y^*_\sigma)_{\sigma\in [\bn]^{\le k}}$), the map
$$\Psi_k: D_k^\omega\to Y, \quad (A,r)\mapsto \sum_{D\preceq A} c_k(|D|,r) y_D,$$
has the following properties:

\begin{align}
\label{E:3.1.2_1} &\Psi_k(b_k^\omega)=0 \text{ and }\Psi_k(t_k^\omega)=2^{k}y_\emptyset \\
\label{E:3.1.2_2} &y^*_\emptyset\big(\Psi_k(A,r)\big)= r2^{k}\text{ for all } (A,r)\in V_k^\omega,\\
\label{E:3.1.2_3} &\Psi_k \text{ is } (1+\vp)-\text{Lipschitz}, 
\intertext{if $x,y\in V_k^\omega$ belong to the same vertical path}
\label{E:3.1.2_4} &\big\| \Psi_k (x)-\Psi_k (y)\big\|_Y\ge d_k(x,y)
\intertext{if $x,y\in V_k^\omega$ do not belong to the same vertical path}
\label{E:3.1.2_5}&\frac{d_k(x,y)}{C(\vp)} \le \big\| \Psi_k (x)-\Psi_k (y)\big\|_Y,\end{align}
where $C(\vp)=6(1+\vp)$ in full generality, and $C(\vp)=3$ if $(y_{A_i})_{i\in\bn}$ is bimonotone. 
\end{claim}

We now proceed with the induction. For $k=1$ we proceed as follows. Consider an $(1+\vp,1+\vp)$-good $\ell_\infty$-tree of height $1$ in $Y$. Thus $\Psi_1: D_1^\omega\to Y$, is the following map 
$$\Psi_1(b_1^\omega)=0, \, \Psi_1(t_1^\omega)=2y_\emptyset, \text{ and }\Psi_1(\{n\},{\scriptstyle \frac12})= y_\emptyset + y_{\{n\}},\text{ for $n\in\bn$}.$$
\eqref{E:3.1.2_1} is clearly satisfied and \eqref{E:3.1.2_2} follows from condition \eqref{E:3.1.1_2}.

For $m<n$ in $\bn$ it follows from \eqref{E:3.1.1_2} that  
$$\big\|\Psi_1(\{n\},\sh)-\Psi_1(\{m\},\sh)\big\|_Y=\|y_{\{n\}}-y_{\{m\}}\|_Y\begin{cases} \le 2 =d_1\big(\{m\},\sh),(\{n\},\sh)\big),\\ \ge |y^*_{\{m\}}(y_{\{n\}}-y_{\{m\}})|=1. \end{cases}$$  
From \eqref{E:3.1.1_3} we deduce that 
$$1\le \big\| \Psi_1(t_1^\omega)-\Psi_1(\{n\},\sh)\big\|_Y=\|y_\emptyset-y_{\{n\}}\|_Y\le 1+\vp.$$
Similiarly we verify that $1\le \big\| \Psi_1(b_1^\omega)-\Psi_1(\{n\},\sh)\big\|_Y\le 1+\vp$.
Since, moreover,  $\|\Psi_1(t_1^\omega)-\Psi_1(b_1^\omega)\|_Y=2\|y_\emptyset\|=2$, we deduce that  \eqref{E:3.1.2_3}, \eqref{E:3.1.2_4} and \eqref{E:3.1.2_5} hold.
 
\medskip

Assume now that Claim \ref{C:3.1} is true for some $k\in\bn$.  Let $\vp>0$ and pick $\eta\le \eta(\frac{\vp}{2})$, to be chosen small enough later, then for all pair of trees $(y_A)_{A\in [\bn]^{\le k}}\subset S_Y$, $(y_A)_{A\in [\bn]^{\le k}}\subset S_{Y^*}$ that form a $(1+\eta,1+\eta)$-good $\ell_\infty^{k+1}$-tree one has that \eqref{E:3.1.2_1}-\eqref{E:3.1.2_5} hold for $\frac{\vp}{2}$.
 
\medskip

Fix now a $(1+\eta,1+\eta)$-good $\ell_\infty^{k+1}$-tree given by $(y_A)_{A\in [\bn]^{\le k+1}}\subset S_Y$ and  
$(y^*_A)_{A\in [\bn]^{\le k+1}}\subset S_{Y^*}$, for some compatible linear ordering $(A_n)_{n\in\bn}$ of $[\bn]^{\le k+1}$.
 For $A=A_m$ and $B=A_n$, we write $A<_{\text{lin}}B$ if $m<n$.

For each $j\in\bn$ consider the trees $(y_{s_j(A)})_{A\in[\bn]^{\le k}}$ and $(y^*_{s_j(A)})_{A\in[\bn]^{\le k}}$ in $S_Y$ and $S_{Y^*}$. In particular $y_{s_j(\emptyset)}=y_{\{j\}}$ and  $y^{*}_{s_j(\emptyset)} =y^*_{\{j\}}$.

Define a binary relation $<_{\text{lin}}^{(j)} $ on $[\bn]^{\le k}$ by $A<_{\text{lin}}^{(j)} B$ if and only if $s_j(A)<_{\text{lin}} s_j(B)$. The following claim is straightforward. 

\begin{claim}
For every $j\in\bn$, the binary relation $<_{\text{lin}}^{(j)}$ is a compatible linear order on $[\bn]^{\le k}$ for which the trees $(y_{s_j(A)})_{A\in[\bn]^{\le k}}$ and $(y_{s_j(A)}^*)_{A\in[\bn]^{\le k}}$, in $S_Y$ and $S_{Y^*}$ respectively, form an $(1+\eta,1+\eta)$-good $\ell_\infty^k$-tree.
\end{claim}

Therefore it follows from our induction hypothesis that for $j\in\bn$ the map 
$$\Psi_k^{(j)}: V_k\to Y, \quad (A,r)\mapsto \sum_{D\preceq A} c_k(|D|, r) y_{s_j(D)},$$ 
satisfies conditions \eqref{E:3.1.2_1}-\eqref{E:3.1.2_5} with 
$(y_{s_j(A)})_{A\in[\bn]^{\le k}}$ and $(y_{s_j(A)})_{A\in[\bn]^{\le k}}$ instead of 
 $(y_A)_{A\in[\bn]^{\le k}}$ and $(y^*_A)_{A\in[\bn]^{\le k}}$, and $\frac{\vp}2$ instead of $\vp$. The map $\Psi_{k+1}$ can actually be written in terms of the maps $\Psi_k^{(j)}$. Indeed, if $A=\{j,a_2,\ldots, a_m\}\kin[\bn]^{\le k+1}\setminus\{\emptyset\}$ and  $r\kin \bb_m$, then 
   
   \begin{align}
   \Psi_{k+1}(A,r)&= \sum_{D\preceq A} c_{k+1}(|D|,r)y_D\notag\\
\label{E:3.1.2_6}&=
   r2^{k+1} y_\emptyset+
   \begin{cases} \sum_{\{j\}\preceq D\preceq A} c_{k}(|D|-1, 2r)y_D &\text{if $0\kle r\kleq \sh$} \\
     \sum_{\{j\}\preceq D\preceq A} c_k(|D|-1, 2(1-r)) y_D&\text{if $\sh \kleq r\kle 1$} 
\end{cases}\\
   &= r 2^{k+1}y_\emptyset +  \begin{cases} \Psi_k^{(j)}\big(s_j^{-1}(A),2r\big) &\text{if $0\kle r\kleq \sh$} \\
                           \Psi_k^{(j)}\big(s_j^{-1}(A),2(1\kminus r)\big) &\text{if $\sh\kle r\kle 1$}. 
\end{cases}\notag\\
 \label{E:3.1.2_7}  &= r 2^{k+1}y_\emptyset +  \begin{cases} \Psi_k^{(j)}\big(I_k^{(j,-)}((A,r))) &\text{if $0\kle r\kleq \sh$} \\
                           \Psi_k^{(j)}\big(F_k^{(j,+)}((A,r))\big) &\text{if $\sh\kle r\kle 1$}. 
\end{cases}
\end{align} 

\textit{Verification of \eqref{E:3.1.2_1} and \eqref{E:3.1.2_2}:}
Elementary computations show that \\$\Psi_{k+1}(b_{k+1}^\omega)=0$, whereas $\Psi_{k+1}(t_{k+1}^\omega)=2^{k+1} y_\emptyset$ follows from \eqref{E:3.1.2_6}.

\medskip

\textit{Verification of \eqref{E:3.1.2_3}:} In order to verify the inequality \eqref{E:3.1.2_3} we can assume that $\{(A,r), (B,s)\}\in E_{k+1}$, and thus, there is a $j\in \bn$ so that $(A,r), (B,s)\in V_{k+1}^{(j,+)}$ or  $(A,r), (B,s)\in V_{k+1}^{(j,-)}$. Secondly, by Lemma \ref{L:2.1} ($\text{b}$)  it follows that $|r-s|=2^{-(k+1)}$, and $|B|=k+1$ and $A\prec B$, or $|A|=k+1$ and $B\prec A$. Let us first assume that  $(A,r), (B,s)\in V_{k+1}^{(j,+)}$, and, without loss of generality that $A\prec B$, and $|B|=k+1$.

\begin{itemize}
\item If $A=\emptyset$, and thus $r=1$, $(A,r)=t_{k+1}^\omega$ and $1-s=2^{-(k+1)}$, it follows from \eqref{E:3.1.2_6}  that
\begin{align}\label{E:3.1.2_8}
\big\|\Psi_{k+1}(A,r)\kminus\Psi_{k+1}(B,s)\big\|_Y&= \big\| (1-s)2^{k+1}y_\emptyset\kminus\Psi_k^{(j)}(F_k^{(j,+)}(B,s))\big\|_Y\\
&\le\!(1\kplus\eta)^2\!\max\big\{ 1\!, \big\|\Psi_k^{(j)}(F_k^{(j,+)}(B,s)\big)\|_Y\big\},\notag
\end{align}
where to get \eqref{E:3.1.2_8} we used first the upper bound in inequality \eqref{E:3.1.1_3}, then the obvious fact that $\sup_i(|a_i|)=\max\{|a_j|,\sup_{i\ne j}|a_i|\}$,  and finally the lower bound in inequality \eqref{E:3.1.1_3}. 
We are now in position to use the induction hypothesis 
\begin{align*}
\big\|\Psi_k^{(j)}(F_k^{(j,+)}(B,s))\big\|_Y&=\big\|\Psi_k^{(j)}(F_k^{(j,+)}(B,s))- \Psi_k^{(j)}(b_{k}^\omega)\big\|_Y\\
&=\big\|\Psi_k^{(j)}(F_k^{(j,+)}(B,s))- \Psi_k^{(j)}(F_k^{(j,+)}(t_{k+1}^\omega)\big\|_Y\\
&\le (1+\frac{\vp}{2})d_k\big(F_k^{(j,+)}(B,s), F_k^{(j,+)}(t_{k+1}^\omega)\big)\\
&= (1+\frac{\vp}{2})d_{k+1}\big((B, s), t_{k+1}^\omega\big)= (1+\frac{\vp}{2}).
\end{align*}

\item If $A\not=\emptyset $, and thus $j=\min A$, we deduce similarly  from \eqref{E:3.1.2_6} and \eqref{E:3.1.1_3}  that 
 \begin{align*}
\big\|\Psi_{k+1}&(A,r)-\Psi_{k+1}(B,s))\big\|_Y\\
&= \big\| (s-r)2^{-(k+1)}y_\emptyset + \Psi_k^{(j)}(F_k^{(j,+)}((A,r)))-\Psi_k^{(j)}(F_k^{(j,+)}((B,s))\big)\|_Y\\
&\le (1\kplus \eta)^2 \max\{1, \|\Psi_k^{(j)}(F_k^{(j,+)}((A,r)))-\Psi_k^{(j)}(F_k^{(j,+)}((B,s))\big)\|_Y\}
\end{align*}
and, using again the induction hypothesis 
\begin{align*}
\big\|\Psi_k^{(j)}(F_k^{(j,+)}((A,r)))-&\Psi_k^{(j)}(F_k^{(j,+)}((B,s))\big)\|_Y  \\
&\le (1\kplus\frac{\vp}{2}) d_k\big(F_k^{(j,+)}(A,r),F_k^{(j,+)}((B,s))\big) \\
&\le (1\kplus\frac{\vp}{2}) d_{k+1}\big((A,r),(B,s)\big)
=1\kplus\frac{\vp}{2}. 
\end{align*}

\end{itemize}
Thus in both cases we obtain 
$$ \big\|\Psi_{k+1}(A,r)-\Psi_{k+1}(B,s))\big\|\le (1+\eta)^2(1+\frac{\vp}{2})\le 1+\vp.$$

The case where $(A,r), (B,s)\in V_{k+1}^{(j,-)}$, for some $j\in\bn$, is obtained in a similar way using the isometry $I_k^{(j,-)}$. 

\medskip

\textit{Verification of \eqref{E:3.1.2_4}:}
We observe that whenever the vertices $(A,r)$ and $(B,s)$ belong to the same vertical path, one has $d_{k+1}\big((A,r),(B,s)\big)=2^{k+1}|r-s|$ but by \eqref{E:3.1.2_2} and \eqref{E:3.1.2_6},
$$2^{k+1}|r-s|= |y^*_\emptyset \big( \Psi_{k+1}(A,r)-\Psi_{k+1}(B,s)\big)|\le  \big\|\Psi_{k+1}(A,r)-\Psi_{k+1}(B,s)\big\|_Y,$$
which proves \eqref{E:3.1.2_4}.

\medskip

\textit{Verification of \eqref{E:3.1.2_5}:}
We introduce  the following renorming of the subspace  $Z=[y_{A_{j}}:j\in\bn]:=\overline{\spa(y_{A_{j}}:j\in\bn)}$ of $Y$. Recall that by assumption
$(y_{A_n})_{n\in\bn}$ is a basis for $Z$ whose basis constant is not greater than $1+\eta$. So 
 for $z=\sum_{n=1}^\infty a_n y_{A_n}$ we put
 $$\tn z\tn =\max_{m\le n} \Big\|\sum_{j=m}^n a_j y_{A_j}\Big\|_Y.$$
It is well known that $(y_{A_i})_{i\in\bn}$ becomes bimonotone with respect to  $\tn \cdot\tn$  and that 
 \begin{equation*}
 \Big\| \sum_{n=1}^\infty a_n y_{A_n}\Big\|_Y \le \bbtn  \sum_{n=1}^\infty a_n y_{A_n} \bbtn\stackrel{(\star)}{\le} 2(1+\eta) \Big\|\sum_{n=1}^\infty a_n y_{A_n}\Big\|_Y.
 \end{equation*}
We will now show inequality \eqref{E:3.1.2_5} for $\tn\cdot\tn$ with $C(\vp)=3$, and \eqref{E:3.1.2_5} for $\|\cdot\|_Y$ will follow from ($\star$). Let $\Yt$ be the Banach space $Z$ with the norm $\tn\cdot\tn$, then $(y_{A_j})_{n\kin\bn}$ is a bimonotone basis of $\Yt$. 

\begin{claim}
The trees $(y_{A})_{A\in[\bn]^{\le k+1}}$ and $(z^*_{A})_{A\in[\bn]^{\le k+1}}$, where $z^*_A=y^*_A|_Z$ for $A\in[\bn]^{\le k}$ are in $S_{\Yt}$ and $S_{\Yt^*}$, respectively, and form an $(1+\eta,1+\eta)$-good $\ell_\infty^k$-tree for $\Yt$.
\end{claim}
The coordinate functionals, which we denote by  $(\yt^*_{A_j})_{j\kin\bn}$ are also a basic bimonotone  sequence in $S_{\Yt^*}$. 

\begin{rem}
It is not necessarily true that $\yt^*_A=z^*_A$ for all $A\in [\bn]^{\le k+1}$, but nevertheless the tree $(\yt^*_{A_j})_{j\kin\bn}$ is a tree in $S_{\Yt^*}$, and form with $(y_{A_j})_{j\in\bn}$ an $(1+\eta,1+\eta)$-good $\ell_\infty^k$-tree for $\Yt$.
\end{rem} 
First of all, for any  $(A,r)\in V_{k+1}$  it follows from \eqref{E:3.1.1_2}, that 
 \begin{align*}
\tn \Psi_{k+1}(t_{k+1}^\omega) -\Psi_{k+1}(A,r)\tn &\ge |\yt^*_\emptyset (\Psi_{k+1}(t_{k+1}^\omega) -\Psi_{k+1}(A,r))|\\
&=2^{k+1}(1-r)=d_{k+1}\big(t_{k+1}^\omega,(A,r)\big), \text{ and }\\
\tn \Psi_{k+1}(A,r)-\Psi_{k+1}(b_{k+1}^\omega)\tn &\ge |\yt^*_\emptyset (\Psi_{k+1}(A,r)-\Psi_{k+1}(b_{k+1}^\omega))|\\
&=2^{k+1}r=d_{k+1}\big((A,r), b_{k+1}^\omega\big).
\end{align*}

Thus, the following five cases remain to be taken care off:
 
\begin{itemize} 
\item If $(A,r),(B,s)\in V_{k+1}^{(j,+)}\setminus\{t_{k+1}^\omega\}$, for some $j\in\bn$, 
 \begin{align*}
 \Psi_{k+1}(A,r)&-\Psi_{k+1}(B,s)=\\ &(r-s)2^{k+1}y_\emptyset+ \Psi_k^{(j)}(F_k^{(j,+)}(A,r))
 &-\Psi_k^{(j)}(F_k^{(j,+)}(B,s)\big).\\
\end{align*}
And thus, using the bimonotonicity and the induction hypothesis, we deduce that 
 \begin{align*}
 \tn\Psi_{k+1}(A,r)-\Psi_{k+1}(B,s)\tn&\ge  \btn\Psi_k^{(j)}(F_k^{(j,+)}((A,r)))-\Psi_k^{(j)}(F_k^{(j,+)}((B,s))\big)\btn\\
 &\ge \frac1{3}d_{k}\big(F_k^{(j,+)}((A,r)),F_k^{(j,+)}((B,s))\big)\\
&=\frac1{3}d_{k+1}\big((A,r)\big), (B,s)\big).\\
 \end{align*}
 
\item With  similar arguments we can show that for $j\in\bn$ and $(A,r),(B,s)\in V^{(j,-)}_{k+1}\setminus\{b_{k+1}^\omega\}$ we have
\begin{align*}
 \tn\Psi_{k+1}(A,r)-\Psi_{k+1}(B,s)\tn&\ge\frac1{3}d_{k+1}\big((A,r)\big), (B,s)\big).
 \end{align*}

\item If $(A,r)\in V_{k+1}^{(i,+)}\setminus\{t_{k+1}^\omega\}$ and  $(B,s)\in V_{k+1}^{(j,+)}\setminus\{t_{k+1}^\omega\}$, with $i\not= j$.

\medskip

Let us first note that since we have already shown \eqref{E:3.1.2_4} and since  $\tn\cdot\tn\ge \|\cdot\|$ it follows that \eqref{E:3.1.2_4} is also satisfied for the norm $\tn\cdot \tn$ instead of $\|\cdot\|$. Let us assume without loss of generality that $s\ge r$.  Since $\yt^*_{\{i\}}$ is the coordinate functional for $y_{\{i\}}$  (in $Y$) it follows from the expression of $\Psi_{k+1}$ in \eqref{E:3.1.2_6}, and the definition of the coefficients that  
\begin{align*}
\tn \Psi_{k+1}(A,r)- \Psi_{k+1}(B,s) \tn &\ge |\yt^*_{\{i\}}\big( \Psi_{k+1}(A,r)- \Psi_{k+1}(B,s))|\\
&=2^k\cdot2(1-r)=2^{k+1}(1-r)\\
&= d_{k+1}((A,r),t_{k+1}^\omega)\\
&\ge \frac12\big( d_{k+1}((A,r),t_{k+1}^\omega\big)+d_{k+1}\big((B,s),t_{k+1}^\omega)\big)\big)\\
&=\frac12d_{k+1}\big((A,r),(B,s)\big).
\end{align*}

\item If $(A,r)\in V_{k+1}^{(i,-)}\setminus\{b_{k+1}^\omega\}$ and  $(B,s)\in V_{k+1}^{(j,-)}\setminus\{b_{k+1}^\omega\}$, with $i\not= j$. This case can be treated with a similar argument, and the same inequality is obtained. 

\item $(A,r)\in V_{k+1}^{(i,+)}\setminus\{t_{k+1}^\omega\}$ and  $(B,s)\in V_{k+1}^{(j,-)}\setminus\{b_{k+1}^\omega\}$, with $i\not= j$.
 
\medskip

We apply the functionals $\yt^*_{\{i\}}$, $\yt^*_{\{j\}}$ and $\yt^*_\emptyset$ to the vector $\Psi_{k+1}(A,r)-\Psi_{k+1}(B,s)$ and obtain
 from \eqref{E:3.1.2_6},
 
\begin{align*}
\big\|\Psi_{k+1}(A,r)-\Psi_{k+1}(B,s)\big\|&\ge |\yt^*_{\{i\}}\big(\Psi_{k+1}(A,r)-\Psi_{k+1}(B,s)\big)|\\
&=2^k\cdot2(1-r)=2^{k+1}(1-r),\\
\big\|\Psi_{k+1}(A,r)-\Psi_{k+1}(B,s)\big\|&\ge |\yt^*_{\{j\}}\big(\Psi_{k+1}(B,s)-\Psi_{k+1}(A,r)\big)|\\
&=2^k2s=2^{k+1}s,\\
\big\|\Psi_{k+1}(A,r)-\Psi_{k+1}(B,s)\big\|&\ge |\yt^*_{\emptyset}\big(\Psi_{k+1}(A,r)-\Psi_{k+1}(B,s)\big)|=2^{k+1}(r-s).
 \end{align*}
 Note that if $r-s\le \frac13$, and $s\le \frac13$, then  $1-r\ge \frac13$, 
 and thus
 $$\big\|\Psi_{k+1}(A,r)-\Psi_{k+1}(B,s)\big\|\ge\frac13  2^{k+1}=\frac13\text{diam}(V_{k+1},d_{k+1})$$
 which implies that 
 $$\big\|\Psi_{k+1}(A,r)-\Psi_{k+1}(B,s)\big\|\ge \frac13 d_{k+1}\big(  (A,r), (B,s) \big),$$
 and finishes the verification of the inequality \eqref{E:3.1.2_5} .
 
\end{itemize}

\end{proof}

\subsection{Banach spaces containing good $\ell_\infty$-trees of arbitrary height}\label{S:4}

The main goal of this section is give sufficient conditions for a Banach space to contain good $\ell_\infty$-trees of arbitrary height almost isometrically. More precisely we will show that if $X$ is a separable reflexive Banach space which has an unconditional asymptotic structure, and with $\Sz(X^*)>\omega$, then $X$ contains good $\ell_\infty$-trees of arbitrary height almost isometrically. The difficulty is to get $(1+\vp,1+\vp)$-good $\ell_\infty$-trees of arbitrary height for every $\vp>0$ rather than merely $(C,D)$-good $\ell_\infty$-trees of arbitrary height for some constant $C,D\ge 1$. A somehow technical and lengthy argument is therefore needed. It can be split into essentially three steps. First we will show that if $X$ is a separable reflexive Banach space with an unconditional asymptotic structure, and if $\Sz(X^*)>\omega$, then for all $n\in\bn$, $\ell_\infty^n$ is in the asymptotic structure of $X$ up to some constant $C\ge 1$. Then we argue, using  an asymptotic analogue of a classical argument of James about the non-distortability of $\ell_\infty$, that for all $n\in\bn$, $\ell_\infty^n$ is in the asymptotic structure of $X$ (i.e. with constant $C$ arbitraily close to $1$). Finally, we prove that if for all $n\in\bn$ $\ell_\infty^n$ is in the asymptotic structure of $X$ then $X$ contains good $\ell_\infty$-trees of arbitrary height almost isometrically.  In Section \ref{S:4.1} we introduce all the ingredients needed to carry out the proof of the main theorem which forms Section \ref{S:4.2}.

\subsubsection{Asymptotic properties and trees}\label{S:4.1}\ \\

\noindent\textbf{Unconditional asymptotic structure.}

\medskip

\noindent We recall the following notion which was introduced in \cite{MMT}. Let $X$ be a Banach space, and let us denote the cofinite dimensional subspaces of $X$ by $\cof(X)$. For $n\in\bn$ let $E$ be an $n$-dimensional Banach space with a normalized basis $(e_j)_{j=1}^n$.
 We say that  {\em $E$  is in the $n$th asymptotic structure of $X$,}  or {\em $X$ contains $E$ asymptotically} and write $E\in \{X\}_n$, if $ \forall \vp>0$, 
 \begin{align}\label{E:4.1.1_1}
\forall X_1\kin \cof(X),\,\exists  x_1&\kin S_{X_1} 
,\ldots,
 \forall X_n\kin \cof(X),\,\exists  x_n\kin S_{X_n} \text{ such that }\\ 
 (x_j)_{j=1}^n &\text{ is $(1+\vp)$-equivalent to $(e_j)_{j=1}^n$.}\notag
\end{align}

For $1\le p\le \infty$, we say that  $\ell_p^n$ {\em is  in the  $n$-th asymptotic structure of $X$,  up to a constant $C\ge1$}, if  for  each 
 $n\in\bn$  there is an $n$-dimensional space $E_n$, with a normalized basis $(e^{(n)}_j)_{j=1}^n$, which is $C$-equivalent to the $\ell_p^n$-unit vector basis, so that 
 $E_n\in\{X\}_n$, for all $n\in\bn$.
 
We will need the following lemma which can be extracted from  \cite[1.8.3]{MMT}

\begin{lemma}\label{L:4.1}
If  $\ell_\infty^n$ is contained in the $n$-asymptotic structure of $X$  up to a constant $C$, then $\ell_\infty^n$ is in the $n$-th asymptotic structure of $X$.
\end{lemma}

\begin{proof}[Sketch of proof] It was observed that for any $E\in\{X\}_n$, with a normalized basis $(e_j)_{j=1}^n$,  and any subspace $F$ of $E$ which is  spanned by a normalized block basis $(f_j)_{j=1}^m$ of  $(e_j)_{j=1}^n$ is in $\{X\}_m$. Secondly James's result on the non distordability of $c_0$ implies that for any $m\in\bn$, $C>1$ and any $\vp$ there is an $n:=n(m,C,\vp)$ so that any $n$-dimensional space $E$ with a normed basis $(e_j)_{j=1}^n$, which is $C$-equivalent to the $\ell_\infty^n$ unit vector basis admits a block basis of length $m$ which is $(1+\vp)$-equivalent to the $\ell_\infty^m$ unit vector basis. The conclusion follows. 
\end{proof}

\begin{rem}
Actually, although not needed in this paper, the same is true for all $1\le p\le \infty$, and follows from the  quantitative version of Krivine's Theorem proven by Rosenthal in \cite[Theorem 3.6]{R}
\end{rem}
 
The following property will be crucially used in the sequel.
\begin{defn}\label{D:2.1} 
We say that a Banach space has an {\em unconditional asymptotic structure with constant $C\ge 1$}, or a $C$-unconditional asymptotic structure, if for all $n\in\bn$
 \begin{align}\label{E:4.1.1_2}
 \exists  X_1\kin \cof(X),\,\forall x_1\kin S_{X_1}& \,
 ,\ldots, 
 \exists X_n\kin \cof(X), \,\forall  x_n\kin S_{X_n}  \text{ so that }\\
(x_j)_{j=1}^n \text{ is $C$-unconditional.}&\notag
 \end{align}
\end{defn}

\begin{rem}
Having and unconditional asymptotic structure for a Banach space is strictly weaker than having an unconditional basic sequence. For instance, the Argyros-Delyianni space \cite{ArgyrosDelyianni1997} does not have any unconditional basic sequence but has an unconditional asymptotic structure. 
\end{rem}

As noted in \cite{MMT}  the property of having an unconditional asymptotic structure or containing a finite dimensional space $E$ asymptotically can be described in the language of a game between two players. Assume that for $n\kin\bn$  the {\em cofinite player} picks a cofinite dimensional subspace $X_1$ and then the  {\em vector player} an element $x_1\in S_{X_1}$. The players repeat this procedure $n$ times to obtain cofinite dimensional subspaces $X_1,X_2,\ldots,X_n$ and and vectors $x_1,x_2,\ldots, x_n$. The cofinite player has won if  the sequence $x_1,x_2,\ldots, x_n$ is a $C$-unconditional sequence. It follows therefore that  $X$ has an $C$-unconditional asymptotic structure if and only if for every 
$n\kin\bn$ the cofinite player has a winning  strategy in that game.
For $n\kin\bn$ an $n$-dimensional space $E$ with a normalized basis $(e_j)_{j=1}^n$  $E$ is in the $n$-dimensional asymptotic structure of $X$  if and only if for every $\vp>0$ the vector player has a winning 
strategy for every $\vp>0$, if his or her  goal is to obtain a sequence $(x_j)_{j=1}^n$ which is $(1+\vp)$-equivalent to $(e_j)_{j=1}^n$. 

\medskip

\noindent\textbf{Tree reformulation for spaces with separable dual.}

\medskip

\noindent In this paper we will only consider trees of finite height which are countably branching, i.e. families of the form $(x_A)_{A\in[\bn]^{\le n}}$ indexed by $[\bn]^{\le n}$, for some $n\in\bn$.

\medskip
 
Let $X$ be a Banach space and let for some $n\in\bn$, $(x_A)_{A\in[\bn]^{\le n}}$ be a tree in $X$. The tree $(x_A)_{A\in[\bn]^{\le n}}$ is said to be {\em normalized } if $(x_A)_{A\in[\bn]^{\le n}}\subset S_X$. A \emph{node} of $(x_A)_{A\in[\bn]^{\le n}}\subset S_X$ is a sequence of the form $\big\{x_{A\cup \{n\}}:n\kin\bn, n>\max(A)\big\}$, where $A\in [\bn]^{< n}$ (\ie $A$ is  not maximal in $[\bn]^{\le n }$). A tree $(x_A)_{A\in[\bn]^{\le n}}$ is a {\em weakly null tree} if all nodes are weakly null sequences. A \emph{branch} of $(x_A)_{A\in[\bn]^{\le n}}$ is a sequence (of length $n+1$) of the form
 $(x_D)_{D\preceq A}$, where $A\in [\bn]^n$ (\ie $A$ is maximal in $[\bn]^{\le n}$). The following definition is a finitary version of the unconditional tree property introduced in \cite{JohnsonZheng2008}.

\begin{defn}
A Banach space $X$ is said to have the $C$-unconditional finite tree property, if for any $n\in\bn$, any normalized weakly null tree $(x_A)_{A\in[\bn]^{\le n}}$ in $X$ has a branch which is $C$-unconditional.
\end{defn}

It is well-known that if a Banach space has a \emph{separable dual}, a property defined in terms of a game can be reformulated in terms of containment of certain countably branching trees whose branches reflect the desired property.  
 
\begin{lemma}
If $X^*$ is separable, then $X$ has a $C$-unconditional asymptotic structure if and only if $X$ has the $C$-unconditional finite tree property.
\end{lemma}

\begin{proof}
In the case that $X^*$ is separable, and the sequence $(x^*_n)_{n\in\bn}\subset S_X$ is dense in $S_X^*$, we can replace in \eqref{E:4.1.1_1} and \eqref{E:4.1.1_2} 
the set $\cof(X)$ by  the  countable set of cofinite dimensional subspaces $(Y_n)_{n\in\bn}$,  with $Y_n=\cN(x^*_1,x^*_2,\ldots x^*_n)=
\{x\in X: x^*_j(x)=0\text{ for $j=1,2,\ldots, n$}\}.$
 From this fact the conclusion follows from \cite[Theorem 3.3]{OdellSchlumprecht2002} or \cite[Corollary 1]{OdellSchlumprecht2006}.
\end{proof}

The next lemma can be proved along the same line and its proof is omitted.

\begin{lemma}\label{L:4.3}
If $X^*$ is separable, then for any $n\in\bn$ and any $n$-dimensional space $E$ with normalized basis $(e_j)_{j=1}^n$, $E$ is in the $n$-th asymptotic structure of $X$ if and only if for every $\vp>0$ there is a weakly null tree  $(x_A)_{A\in[\bn]^{\le n}}$ in $S_X$ all of  whose branches are $(1+\vp)$-equivalent to $(e_j)_{j=1}^n$. Moreover, in that case $(x_A)_{A\in[\bn]^{\le n}}$ can be chosen inside every cofinite dimensional subspace.
\end{lemma}

\subsubsection{Sufficient conditions for the containment of good $\ell_\infty$-trees}\label{S:4.2}

In this section we give sufficient conditions that guarantee the presence of good $\ell_\infty$-trees of arbitrary height. The good $\ell_\infty$-trees of arbitrary height are obtained by pruning carefully certain trees. The procedure of \emph{pruning} a tree is what corresponds to the action of taking a subsequence for a sequence. Formally speaking we define a pruning of $(x_A)_{A\in[\bn]^{\le n}}$ as follows. Let $\pi:[\bn]^{\le n}\to [\bn]^{\le n}$ be an order isomorphism with the property that if $F\in[\bn]^{< n}$, then for any $n\in \bn$, so that $n>\max(A)$ and $A\cup\{n\}\in [\bn]^{\le n}$, $\pi(A\cup\{n\})$ is of the form $\pi(A)\cup \{ s_n\}$, where $(s_n)_{n\in\bn}$ is a sequence which  increases with $n$. The family $(x_A)_{A\in\pi([\bn]^{\le n})}$ is called a \emph{pruning} of $(x_A)_{A\in[\bn]^{\le n}}$. Let $\xt_A:=x_{\pi(A)}$ for $A\in [\bn]^{\le n}$, then $(\xt_A)_{A\in [\bn]^{\le n}}$ is simply a \emph{relabeling} of the family $(x_A)_{A\in \pi([\bn]^{\le k})}$, and is also called a pruning of $(x_A)_{A\in[\bn]^{\le n}}$. It is important to note that the branches of a pruning of a tree are a subset of the branches of the original tree, which follows from the observation that for any $B\in [\bn]^{\le n}$, $\{A\in  [\bn]^{\le n}: A\prec \pi(B)\}=
  \{\pi(A): A\prec B\}$. We also note that the nodes of a pruned tree are subsequences of the nodes of the original tree. 
  
\medskip

Let us finally mention, how we usually choose prunings inductively. Let $(A_i)_{i\in\bn}$ be a compatible linear order of $[\bn]^{\le n}$, recall that it means that for any $i,j$ in $\bn$, if $\max(A_i)<\max(A_j)$ then $i<j$. It follows that if $A_i\prec A_j$, then $i<j$, and if $A_i=A\cup\{s\}\in [\bn]^{\le n}$ and $A_j=A\cup\{t\}\in[\bn]^{\le n}$, for some (non maximal) $A\in[\bn]^{<n}$ and $s<t$ in $\bn$, then $i<j$. Necessarily $A_1=\emptyset$, so we must choose $\pi(A_1)=\emptyset$. Assuming now that $\pi(A_1)$, $\pi(A_2), \ldots ,\pi(A_i)$ have been chosen, then $A_{i+1} $ must be of the form $A_{i+1}= A_l\cup \{m\} $, with $1\le l\le i$ and $m>\max(A_l)$. Moreover if, $m>\max (A_l)+1$ and if $A_l\cup\{m-1\}\in [\bn]^{\le n}$ then $A_l\cup\{m-1\}= A_j$ with $l<j<i+1$, and $\pi(A_j)= \pi(A_l)\cup\{s\}$ for some $s$  which has already been chosen. Thus, we need to choose $\pi(A_{i+1})$ to be of the form $\pi(A_l)\cup\{t\}$, where we require that $t>\max(\pi(A_l))$, and, if  $m>\max(A_l)+1$  then we also need $t>s$, where $s$ is defined as above.

We now proceed with the study of the containment of good $\ell_\infty$-trees of arbitrary height almost isometrically in Banach spaces. For the convenience of the reader we reproduce below the definition of good $\ell_\infty$-trees of arbitrary height.

\medskip

\noindent\textbf{Definition 3.1}
We say that a Banach space $X$ contains $(C,D)$-good $\ell_\infty$-trees of arbitrary height if there are constants $C,D\ge 1$ such that for any $n\in \bn$, 
and any compatible linear ordering $(A_i)_{i\in \bn}$  of $[\bn]^{\le n}$ there are a vector-tree $(x_A)_{A\in[\bn]^{\le n}}$ in $S_X$  and a functional-tree $(x^*_A)_{A\in[\bn]^{\le n}}$ in $S_{X^*}$
satisfying the following properties:\\
(16) for all $A,B\in [\bn]^{\le n}$, with $\max(B)>\max(A)$,
\begin{equation*}
 x^*_A(x_A)=1 \text{ and } x^*_A(x_B)=0,
\end{equation*}
(17) for all $(a_i)_{i=0}^n\subset \br$, and all $B\in [\bn]^n$,
\begin{equation*}
 \frac{1}{C}\|(a_i)_{i=0}^n\|_\infty\le\|a_0x_\emptyset +a_1 x_{\{b_1\}}+\cdots +a_n x_{\{b_1,\dots, b_n\}}\|_X \le C\|(a_i)_{i=0}^n\|_\infty,
\end{equation*}
(18) for every $i\le j$, every $(a_m)_{m=0}^j\subset \br$, one has
\begin{equation*}
\|\sum_{m=1}^i a_m x_{A_m}\|_X\le D\|\sum_{m=1}^j a_m x_{A_m}\|_X,
\end{equation*}


If $X^*$ is separable and if $\ell_\infty^n$ is in the $n$-th asymptotic structure of $X$ up to a constant $D\ge 1$, then by Lemma \ref{L:4.3} there is a normalized weakly tree all of  whose branches are $D(1+\vp)$-equivalent to the canonical basis of $\ell_\infty^{n+1}$. Obviously, such a tree (and any of its prunings) satisfies condition \eqref{E:3.1.1_3}, but \eqref{E:3.1.1_2} and \eqref{E:3.1.1_4} might not hold. In the next lemma we show that for any compatible linear ordering there exists a pruning such that \eqref{E:3.1.1_2} and \eqref{E:3.1.1_4} hold.

\begin{lemma}\label{L:4.4}  Let $X^*$ be separable. Assume that for all $n\in\bn$, $\ell_\infty^n$ is in the $n$-th asymptotic structure of $X$ up to a constant $D\ge 1$. Then for each $\vp>0$, $X$ contains $(D(1+\vp),1+\vp)$-good $\ell_\infty$-trees of arbitrary height. In particular, if $\ell_\infty^n$ is in the $n$-th asymptotic structure of $X$, then $X$ contains good $\ell_\infty$-trees of arbitrary height almost isometrically.
\end{lemma}

\begin{proof} 
 
Let $n\in\bn$, $\vp>0$, and let  $(A_j)_{j\in\bn}$ be a compatible linear ordering of $[\bn]^{\le n}$. Since $\ell_\infty^n$ is in the $n$-th asymptotic structure of $X$ up to a constant $D\ge 1$ and $X^*$ is separable, by Lemma \ref{L:4.3} there is a weakly null tree $(y_A)_{A\in[\bn]^{\le n}}$ in $X$ such that $\|y_A\|=1$ all of  whose branches are $D(1+\vp)$-equivalent to the canonical basis of $\ell_\infty^{n+1}$. Moreover, $(y_A)_{A\in[\bn]^{\le n}}$ can be chosen inside every cofinite dimensional subspace. Note that \eqref{E:3.1.1_3} is satisfied for the tree $(y_A)_{A\in[\bn]^{\le n}}$. We choose $0<\delta<\vp$ small enough so that the following condition holds:
  \begin{align}\label{E:4.2.2_1} &\text{If $(v_j)_{j=1}^{n+1}$ is a normalized basic sequence, whose basis constant}\\
 &\text{is not larger than $(1+\vp)$, then any sequence $(u_i)_{i=1}^{n+1}$, for which}\notag\\
   &\text{$\|u_j-v_j\|\le \delta$, $j=1,2\ldots, n$, is $\sqrt{1+\vp}$-equivalent to $(v_j)_{j=1}^{n+1}$. }\notag
  \end{align}
 
 $X^*$ being separable, pick $(f^*_j)_{j\in \bn}$ a dense sequence in $S_{X^*}$.
 Inductively we will choose  for every $j\in\bn$, an element $\At_j\in [\bn]^{\le n}$, a finite set $F_j\subset S_{X^*}$, and an element $\yt_j\in B_X$ so that
 \begin{align}
\label{E:4.2.2_2}  &\At_1=\emptyset, \text{ and if }A_j=A_i\cup\{\max(A_j)\}, \text{ for some $i<j$, then}\\ 
           &\text{$\At_j = \At_i\cup\{ m\}$, with $m>  \max\{\At_s:s<j\}$},\notag\\
\label{E:4.2.2_3}  &\{f^*_1,f^*_2,\ldots, f^*_j\}\subset F_j\text{ and there is an $\yt_j^*\in F_j$, with $\yt^*_j(\yt_{j})=\|\yt_{j}\|$,}\\
\label{E:4.2.2_4}  &\max_{f^*\in F_j} |f^*(x)|\ge \frac1{1+\delta}\|x\|\text{ for all $x\in \spa(\yt_i)_{i=1}^j$,}\\
\label{E:4.2.2_5} &f^*(\yt_j)=0\text{ for all $f^*\in F_i$, with $i<j$,}\\
\label{E:4.2.2_6} &\|\yt_j-y_{\At_j}\|<\delta,\text{ and } \|\yt_j\|=1.
\end{align}

For  $j=1$, we put $\yt_1= y_\emptyset$, choose $\yt_1^*\in S_{X^*}$ so that $\yt^*_1(\yt_1)=1$, and put $F_1=\{f_1^*, \yt^*_1\}$, \eqref{E:4.2.2_2}-\eqref{E:4.2.2_6} is then clearly satisfied. Now assume  that for some $j\in\bn$  we have chosen $\At_i$, $F_i$ and $\yt_i$, for $i=1,2,\ldots, j$. First we note that \eqref{E:4.2.2_4} and \eqref{E:4.2.2_5} implies that the normalized sequence $(\yt_i)_{i=1}^j$ is a basic sequence which constant not exceeding $1+\delta$. Since the linear ordering is compatible we can write $A_j$ as $A_j=A_{i_0}\cup\{ t\}$ with $i_0<j$ and  $t>\max(A_{i_0})$ and let $m_0=\max_{s<j} \At_s$, and put  
\begin{equation*}
Y_j=\Big\{ x\in X: f^*(x)=0 \text{ for all $f^*\in \bigcup_{i=1}^{j} F_i$}\Big\}\in \cof(X).
\end{equation*} 
Since the tree $(y_A)_{A\in [\bn]^{\le n}}$ is normalized and weakly null, we can find $m>m_0$ large enough so that $\dist\Big(S_{Y_j},y_{\At_{i_0}\cup\{m\}} \Big)<\delta$, and thus there is an $\yt_{j+1}\in Y$ with $\|y_{\At_{i_0}\cup\{m\}}-\yt_{j+1}\|<\delta$ and $\|\yt_{j+1}\|=\|y_{\At_{i_0}\cup\{m\}}\|=1$. Letting $\At_{j+1}= \At_{i_0}\cup\{m\}$, we deduce \eqref{E:4.2.2_2} and \eqref{E:4.2.2_6}, while \eqref{E:4.2.2_5} follow from the definition of $Y_j$ and the requirement that $\yt_{j+1}\in Y_j$. Finally we choose a finite set  $F_{j+1}\subset S_{X^*}$ so that \eqref{E:4.2.2_3} and \eqref{E:4.2.2_4} hold. This finishes the inductive step of choosing $\At_j\in [\bn]^{\le n}$, a finite set $F_j\subset S_{X^*}$,  and an element $\yt_j$, for every $j\in\bn$.

\medskip

We first  deduce from \eqref{E:4.2.2_2}, that the map $\pi: [\bn]^{\le n}\to [\bn]^{\le n}, A_j\mapsto\At_j$ is an order isomorphism, that satisfies the condition for prunings, and thus, by \eqref{E:4.2.2_6}, $(\yt_j)_{j\in \bn}$ is a $\delta$-perturbation of $(y_{\At_j})_{j\in\bn}$, which is  a pruning of $(y_{A})_{A\in[\bn]^{\le n}}$. We define for $A\in[\bn]^{\le n}$, $z_A=\yt_{j}$, and $z^*_A=\yt^{*}_j$, where $j\in\bn$ is such that $\pi(A)=\At_j$. First we deduce from \eqref{E:4.2.2_4} and \eqref{E:4.2.2_5} that $(z_{A_i})_{i=1}^\infty$ is basic and that its basis constant (which is the same as the basis constant of $(\yt_i)_{i\in\bn}$)
is at most $1+\delta\le 1+\vp$, which verifies \eqref{E:3.1.1_4}. Condition \eqref{E:3.1.1_2} follows from the second condition in \eqref{E:4.2.2_3} and \eqref{E:4.2.2_5}. This conclude the proof of our lemma since, as we already mentioned, condition \eqref{E:3.1.1_3} follows from our assumption about the tree $(y_A)_{A\in[\bn]^{\le n}}$.
\end{proof}

We shall now introduce a crucial tool, the so-called {\it Szlenk index}. The definition of the Szlenk derivation and the Szlenk index have been first introduced in \cite{Szlenk1968} and we refer to \cite{Lancien2006} or \cite{OdellSchlumprecht2006} for a thorough account on this central notion in asymptotic Banach space theory. Consider a real separable Banach space $X$ and $K$ a weak$^*$-compact subset of
$X^*$. For $\vp>0$ we let $\cV$ be the set of all relatively weak$^*$-open
subsets $V$ of $K$ such that the norm diameter of $V$ is less than $\vp$ and
$s_{\vp}(K)=K\setminus \cup\{V:V\in\cV\}.$ We define inductively
$s_{\vp}^{\alpha}(K)$ for any ordinal $\alpha$, by
$s^{\alpha+1}_{\vp}(K)=s_{\vp}(s_{\vp}^{\alpha}(K))$ and
$s^{\alpha}_{\vp}(K)={\displaystyle \cap_{\beta<\alpha}}s_{\vp}^{\beta}(K)$ if
$\alpha$ is a limit ordinal.  Then we define $\Sz(X,\vp)$ to be the least
ordinal $\alpha$ so that $s_{\vp}^{\alpha}(B_{X^*})=\emptyset,$ if such an
ordinal exists. Otherwise we write $\Sz(X,\vp)=\infty.$ The {\it Szlenk index}
of $X$ is finally defined by $\Sz(X)=\sup_{\vp>0}\Sz(X,\vp).$ We denote $\omega$ the first infinite ordinal and $\omega_1$ the first uncountable ordinal. The Szlenk index, in particular its utilization in the next lemma, is the pivotal notion in this article since it allows us to establish a bridge between the embeddabilty results from Section \ref{S:3} and the non-embeddability results from Section \ref{S:5}. Indeed, it follows from a theorem of Knaust, Odell and Schlumprecht \cite{KnaustOdellSchlumprecht1999} that a separable Banach space
admits an equivalent asymptotically uniformly smooth norm if and only if $\Sz(X)\le \omega$. Then it is easy to see that for a reflexive Banach space the condition $\Sz(X^*)\le \omega$ is
equivalent to the existence of an equivalent asymptotically uniformly convex norm on $X$. With this information at hand, we can almost forget the formulations in terms of renormings and work essentially with the notion of the Szlenk index of a Banach space.

\medskip

To prove Theorem \ref{T:4.1} below it would be sufficient to show that a separable reflexive Banach space $X$, that has a $C$-unconditional asymptotic structure for some $C\ge 1$ and such that $\Sz(X^*)>\omega$, contains $\ell_\infty^n$ in its $n$-th asymptotic structure up to some constant. Modifying slightly the proof of Lemma \ref{L:4.4} one can actually prove a little bit more.

\begin{lemma}\label{L:4.5} Let $X$ be a separable reflexive Banach space. Assume that $X$ has a $C$-unconditional asymptotic structure for some $C\ge 1$ and that $\Sz(X^*)>\omega$. Then, there is $\eta>0$ so that for all $\vp>0$, $X$ contains $(\sqrt{1+\vp}\frac{4C}\eta,1+\vp)$-good $\ell_\infty$-trees of arbitrary height. In particular, there is $\eta>0$ so that for all $\vp>0$ and every $n\in \bn$, $\ell_\infty^n$ is in the $n$-th asymptotic structure of $X$ up to constant $\sqrt{1+\vp}\frac{4C}\eta$.
\end{lemma}

\begin{proof} 
Without loss of generality (simply pass to an appropriate cofinite dimensional subspace  $X_1$ if needed) we  can assume that 
 \begin{align}\label{E:4.2.2_7}
\forall x_1\kin S_{X},\, \exists X_2\kin \cof(X),& \,
 ,\ldots, 
 \exists X_{n+1}\kin \cof(X), \,\forall  x_{n+1}\kin S_{X_{n+1}}  \text{ so that }\\
(x_j)_{j=1}^{n+1} \text{ is $C$-unconditional.}&\notag
 \end{align}
Since $X$ is separable and reflexive with $\Sz(X^*)>\omega$, we can apply \cite[Theorem 5.1, (i)$\Rightarrow$ (iii)]{MotakisSchlumprecht} to the space $X^*$, and conclude, that there is an $\eta>0$,  so that for each $k\in\bn$ there  is  a tree  $(x^{(k)}_A)_{A\in [\bn]^{\le k}}\subset B_X$ and a weakly null tree  $(x^{(k)*}_A)_{A\in [\bn]^{\le k}}\subset B_{X^*}$ with the following properties:
\begin{itemize} 
\item $x^{(k)*}_A(x_B^{(k)})>\eta$, for all $A,B\in[\bn]^{\le k}\setminus\{\emptyset\}$, with $A\preceq B $,
\item $\big|x^{(k)*}_{A}(x_B^{(k)})\big| < \eta/2$, for all $A,B\in [\bn]^{\le k}\setminus\{\emptyset\}$,  with $A\not\preceq B$,
\item w-$\lim_{n\to\infty } x_{A\cup\{n\}}^{(k)}= x_A^{(k)}$, for all $A\in [\bn]^{< k}$.
 \end{itemize}  
 We first choose a tree $(y_A)_{A\in[\bn]^{\le n}}$ in $X$ as follows. We put 
 \begin{equation}\label{E:4.2.2_8} 
 y_\emptyset =\frac12 x^{(n)}_\emptyset,\text{ and }y_A=\frac12\big(x^{(n)}_A- x^{(n)}_{A\setminus\{\max(A)\}}\big) \text{ for $A\in [\bn]^{\le n}\setminus\{\emptyset\}$}.
 \end{equation}
It follows that $(y_A)_{A\in [\bn]^{\le k}}$ is weakly null and that $\frac\eta4 \le \|y_A\|\le 1$ for all $A\in [\bn]^{\le n}$. 

\medskip

Let $n\in\bn$, $\vp>0$, and let $(A_j)_{j\in\bn}$ be a compatible linear ordering of $[\bn]^{\le n}$. Because we are now dealing with semi-normalized trees, we modify slightly condition \eqref{E:4.2.2_1} as follows. We choose $0<\delta<\vp$ small enough so that the following condition holds: 

\smallskip

\noindent $(47')$ If $(v_j)_{j=1}^{n+1}$ is a basic sequence with $\|v_j\|\in[\eta/4,1]$, for $j=1,2,\ldots,n+1$, whose basis constant is not larger than $(1+\vp)$, then any sequence $(u_i)_{i=1}^{n+1}$, for which $\|u_j-v_j\|\le \delta$ for every $j=1,2\ldots, n+1$, is $\sqrt{1+\vp}$-equivalent to $(v_j)_{j=1}^{n+1}$.

\smallskip

With a similar argument we can choose inductively, for every $j\in\bn$, an element $\At_j\in [\bn]^{\le n}$, a set $F_j$, an element $\yt_j\in B_X$ satisfying $\eqref{E:4.2.2_2}$-\eqref{E:4.2.2_5},  and 
\begin{align*}
\hskip -3cm(52') \hskip 2cm \|\yt_j-y_{\At_j}\|<\delta\|y_{\At_j}\|, \text{ and  }\|\yt_j\|\in[\eta/4,1],
\end{align*}
and some crucial additional condition. Before stating this last condition we note that from \eqref{E:4.2.2_2} it follows that $I_j=\{s: A_s\preceq A_j\}\subset\{1,2,\ldots, j\}$ and we put $l:=|I_j|=|A_j|$. The last condition is:
\begin{align}\label{E:4.2.2_9}
\exists X_{l+1}\kin \cof(X),\, \forall y_{l+1}\kin S_{X_{l+1}},\, 
\ldots, \exists X_{k}\kin \cof(X),\, &\forall y_{k}\kin S_{X_{k}} \text{ so that  }\\
\{ \yt_s: s\in I_j\}\cup\{y_{l+i}: i=1,2\ldots, k-l\} &\text{ is $C$-unconditional}.\notag
 \end{align}
Note that if $l=k$, then \eqref{E:4.2.2_9} means that  $\{ \yt_s: s\in I_j\}$ is $C$-unconditional. 

\smallskip

The base case is straightforward. Now assume that for some $j\in\bn$  we have chosen $\At_i$, $F_i$ and $\yt_i$, for $i=1,2,\ldots, j$. Using \eqref{E:4.2.2_9} we can find a space $\Xt\in\cof(X)$ so that  (with $I_j$ and $l$ as defined before \eqref{E:4.2.2_9})
$\forall y_{l+1}\kin S_{\Xt},\, \exists X_{l+2}\kin \cof(X),\,\forall y_{l+2}\kin S_{X_{l+2}},\, \ldots, \exists X_{k}\kin \cof(X),\, \forall y_{k}\kin S_{X_{k}}$
 and $\{ \yt_s: s\in I_j\}\cup\{y_{l+i}\colon i=1,2\ldots k-l\}$ is $C$-unconditional. If $l=n$, we simply put $\Xt=X$. We define, $\At_{j+1}$ similarly using 
\begin{equation*}
Z_j=\Xt\cap \Big\{ x\in X: f^*(x)=0 \text{ for all $f^*\in \bigcup_{i=1}^{j} F_i$}\Big\}\in \cof(X),
\end{equation*}
instead of $Y_j$, and thus \eqref{E:4.2.2_9} follow from the definition of $Z_j$ and the requirement that $y_{j+1}\in Z_j$. This finishes the inductive step.

\medskip

We define for $A\in[\bn]^{\le n}$, $z_A=\frac{\yt_{j}}{\|\yt_{j}\|}$, and $z^*_A=\yt^{*}_j$, where $j\in\bn$ is such that $\pi(A)=\At_j$. It remains to prove that \eqref{E:3.1.1_3} holds. We first note that for any $B$ in $[\bn]^{n+1}$, say $B=A_j$, it follows that $(z_E)_{E\preceq B}=\Big(\frac{\yt_i}{\|\yt_i\|}\Big)_{i\le j,\,A_i\preceq A_j}$ is $(1+\vp)$-unconditional by \eqref{E:4.2.2_9}. Letting $I=\{i\le j: A_i\preceq A_j\}$, we deduce 
 \begin{align*}
\max\Big\{ &\Big\|\sum_{E\preceq B} \xi_E x_E\Big\|: (\xi_E)_{E\preceq B}\subset [-1,1]\Big\}\\
&\le\frac4\eta  \max\Big\{ \Big\|\sum_{i\in I} \xi_i \yt_i\Big\|: (\xi_i)_{i\in I}\in [-1,1]^I\Big\} \text{ (since $\|\yt_i\|\ge \eta/4$, $i\in I$)}\\
&= \frac4\eta  \max\Big\{ \Big\|\sum_{i\in I} \xi_i \yt_i\Big\|: (\xi_i)_{i\in I}\in \{-1,1\}^I\Big\} \text{ (by convexity we only}\\
&\text{need to consider the extreme points of $[-1,1]^I$ which are $\{-1,1\}^{I}$)}\\
&=\frac{4C}\eta \Big\|\sum_{i\in I} \yt_i\Big\|  \text{ (by \eqref{E:4.2.2_9})}\\
&\le \sqrt{1+\vp}\frac{4C}\eta\Big\|\sum_{E\preceq \pi(B)}  y_E\Big\| \text{ (by \eqref{E:4.2.2_1} and \eqref{E:4.2.2_6})}\\
&= \sqrt{1+\vp}\frac{4C}\eta \frac12 \|x^{(n+1)}_{\pi(B)}\| \text{ (by \eqref{E:4.2.2_8})}\\
&\le \sqrt{1+\vp}\frac{4C}\eta.
 \end{align*}
Since $(z_E)_{E\preceq B}$  is $(1+\vp)$-unconditional it also follows that
$$ \Big\|\sum_{E\preceq B} \xi_E z_E\Big\|\ge \frac1{1+\vp} \max_{E\preceq B} |\xi_E|, \text{ for all   $(\xi_E)_{E\preceq B}\subset [-1,1]$,}$$
which finishes the verification of \eqref{E:3.1.1_3} and thus the proof of the lemma.
\end{proof}

\begin{theorem}\label{T:4.1}
Let $X$ be a separable reflexive Banach space. Assume that $X$ has an unconditional asymptotic structure, and that $\Sz(X^*)>\omega$, then $X$ contains good $\ell_\infty$-trees of arbitrary height almost isometrically. 
\end{theorem}

\begin{proof}  We will apply Lemma \ref{L:4.5} and Lemma \ref{L:4.4} successively as follows. First note that, by Lemma \ref{L:4.5}, for all $n\in\bn$ $\ell_\infty^n$ is in the $n$-th asymptotic structure of $X$ up to some constant $D$. By Lemma \ref{L:4.1}, for all $n\in\bn$, it follows  that $\ell^n_\infty$ is in the $n$-th asymptotic structure  of $X$. An appeal to Lemma \ref{L:4.4} finishes the proof.
 \end{proof}

\subsection{Embeddability into $L_1$}\label{S:3.2}

The fact that the countably branching diamond of depth $1$ embeds isometrically into $L_1$ is well known and can be found implicitly in Enflo's (unpublished) argument showing that $\ell_1$ and $L_1[0,1]$ are not uniformly homeomorphic (cf. \cite{Benyamini1985}, \cite{BenyaminiLindenstrauss2000}, or \cite{Weston1993}). The embedding is based on the particular behavior of Rademacher functions in $L_1[0,1]$. Whether or not a \emph{similar} embedding, for countably branching diamond graphs of arbitrary depth, can be implemented without blowing up the distortion is not completely obvious at first sight. However, using the results from \cite{GNRS2004} and an ultraproduct argument it can be shown that $D_k^\omega$ embeds into $L_1$ with distortion at most $2$. In this section, an embedding using Bernoulli random variables (and no ultraproduct argument) that achieves the same distortion is given. 
Note that it follows from \cite{LeeRaghavendra2010} that this distortion is optimal. It is worth noticing that in both our embedding proofs we start out with a tree $(x_A)_{A\in[\bn]^{\le k}}$, and in both cases the definition of of the image of $(A,r)$ only depends on  $\{x_{A|_i}: i=1,2, \ldots, |A|\}$. We start with the following technical lemma.

\begin{lemma}\label{L:3.1} Let $(\Omega,\Sigma,\bp)$ be an atomless probablity space. For every $k\in\bn$, and every $x\in V_k^\omega$, there exist measurable sets $S_k(x)\subset \Omega$ such that
 
\begin{equation}\label{E:3.2_1}
S_k(x)\subset S_k(y)  \text{ whenever $x,y\kin V_k^\omega$ lie on the same vertical path,}
\end{equation}
and
\begin{equation}\label{E:3.2_2}
S_k(x)\text{ and }S_k(y)\text{ are independent if } \min(x)\neq\min(y).
\end{equation}

\end{lemma}
\begin{proof}
We are using the non recursive description of the diamond graph $D_k^\omega=(V_k^\omega,E_k^\omega)$. Let $k\in\bn$ be fixed. We consider a (countable) family $\{\vp_A:A\kin[\bn]^{\le k}\setminus\{\emptyset\}\}$, indexed  by the elements of $[\bn]^{\le k}\setminus\{\emptyset\}$, of independent Bernoulli random variables,  \ie
 $\bp(\vp_A\keq1)=\bp(\vp_A\keq-1)=1/2$. For $(A,r)\in V_k^\omega$, we will define a  measurable set $S_k(A,r)\subset \Omega$ as follows.
We put $S_k(\emptyset,0)=\emptyset$ and $S_k(\emptyset,1)=\Omega$.
For $(A,r)\in V_k^\omega$,  with $A\not=\emptyset$, say $A=\{a_1, a_2,\ldots ,a_n\}$ and $r=\sum_{i=1}^n \sigma_i 2^{-i}\in \bb_n$, (\ie $\sigma_i\in\{0,1\}$, $i=1,2\ldots, n-1$ and $\sigma_n=1$) we proceed as follows. For $1\le i\le n$, so that $\sigma_i=1$, define
\begin{equation}\label{E:3.2_3}T_k^i(A,r):=\{\vp_{A|_i}=1\}\cap\bigcap_{m<i, \sigma_m=1} \{\vp_{A|_m}=-1\}\cap  \bigcap_{m<i, \sigma_m=0} \{\vp_{A|_m}=1\}.\end{equation}
Then we put
\begin{equation}\label{E:3.2_4}
S_k(A,r):=\bigcup_{i=1,\sigma_i=1}^n T_k^i(A,r).
\end{equation}
Let us make some easy observations. Since for $1\le i<j\le n$ with $\sigma_{i}=\sigma_j=1$, it follows that
\begin{equation*}
T_k^i(A,r)\subset \{\vp_{A|_i}=1\} \text{ and }T_k^j(A,r)\subset\{ \vp_{A|_i}=-1\},
\end{equation*}
it follows that  the $T_k^i(A,r)$'s are pairwise disjoint for $1\le i\le n$, with $\sigma_i=1$.
Secondly from the independence of the sequence $(\vp_{A|_i})_{i=1}^n$ it follows that 
\begin{equation}\label{E:3.2_5}
\bp( T_k^i(A,r))=2^{-i},\text{ whenever $1\le i\le n$, with $\sigma_i=1$}.
\end{equation}
It follows therefore that
\begin{equation}\label{E:3.2_6}
\bp(S_k(A,r))=\sum_{i=1,\sigma_i=1}^n \bp(T_k^i(A,r))=\sum_{i=1,\sigma_i=1}^n 2^{-i}=r. 
\end{equation}
Next assume that   $(A,r),(B,s)\in V_k^\omega$ and $\{(A,r),(B,s)\}\in E_k^\omega$. As noticed previously it follows that $r=s\pm 2^{-k}$. We claim that $S_k(A,r)\subset S_k(B,s)$, if $r=s-2^{-k}$, and that $S_k(B,s)\subset S_k(A,r)$ if $r=s+2^{-k}$. In order to show our claim we can, as noted before, assume that $B\in[\bn]^k$, $s=\sum_{i=1}^k \sigma_i 2^{-i}\in \bb_k$, and that one of the two following  cases holds:
 
 \noindent{\bf Case 1.}
  $r=s-2^{-k}$ and 
 $$r=\sum_{i=1}^{i^-} \sigma_i 2^{-i},\text{ and } A=B|_{i^-} \text{ with  $i^-=\max\{1\le i\le k-1\colon\sigma_i=1\}$}$$
if that maximum exists and otherwise $A=\emptyset$ and $r=0$.

\noindent{\bf Case 2.} $r=s+2^{-k}$  and 
$$r=\sum_{i=1}^{i^+-1} \sigma_i 2^{-i}+2^{-i^{+}},\text{ and } A=B|_{i^+}\text{ with $i^+=\max\{1\le i\le k-1\colon\sigma_i=0\}$}$$
if that maximum exists and otherwise  $A=\emptyset$ and $r=1$.

In order to show our claim in the first case we can assume that $i^-$ exists, since otherwise $S_k(A,r)=\emptyset$.
Since for every $i\le i^-$, for which  $\sigma_i=1$, we deduce that  
  \begin{align*}
  T_k^i(A,r)&=\{\vp_{A|_i}=1\}\cap\bigcap_{m<i, \sigma_m=1} \{\vp_{A|_m}=-1\}\cap  \bigcap_{m<i, \sigma_m=0} \{\vp_{A|_m}=1\}\\
                 & =\{\vp_{B|_i}=1\}\cap\bigcap_{m<i, \sigma_m=1} \{\vp_{B|_m}=-1\}\cap  \bigcap_{m<i, \sigma_m=0} \{\vp_{B|_m}=1\}\\
                  &=T_k^i(B,s),
                 \end{align*}
we deduce our claim in that case.

In the second case we can similarly  observe that $ T_k^i(A,r)=T_k^i(B,s)$ if $i<i^+$. By definition of $i^+$, $\sigma_{i^+}=0$ and $\sigma_i=1$ for all $i^+<i\le k$. Thus we note that for all $i^+<i\le k$
\begin{align*}
T_k^i(B,s)&=\\
 \{\vp_{B|_i}=1\}\cap &\bigcap_{m=1, \sigma_m=1}^{i^+-1}\{ \vp_{B|_m}=-1\}\cap \bigcap_{i^+<m<i}\{ \vp_{B|_i}=-1\}\cap \bigcap_{m=1, \sigma_m=0}^{i^+}\{ \vp_{B|_m}=1\}\\
&\subset\bigcap_{m<i^+, \sigma_m=1}\{ \vp_{B|_m}=-1\}\cap \bigcap_{m< i^+, \sigma_m=0}\{ \vp_{B|_m}=1\}\cap\{\vp_{B|_{i^+}}=1\} \\
& =\{\vp_{A|_{i^+}}=1\}\cap   \bigcap_{m<i^+, \sigma_m=1}\{ \vp_{A|_m}=-1\}\cap \bigcap_{m< i^+, \sigma_m=0}\{ \vp_{A|_m}=1\}\\
&=T_k^{i^{+}}(A,r),
\end{align*}
which also implies our claim in the second case.

An easy consequence of our last observation is that $S_k(A,r)\subset S_k(B,s)
$ whenever $(B,s)$ lies above $(A,r)$ on the same vertical path. From the independence condition of the family $(\vp_A)_{A\in[\bn]^{\le k}}$, and since for every $(A,r)\in V_k^\omega$, with $A\not=\emptyset$,
 the set  $S_k(A,r)$ is in the $\sigma$-algebra generated by  the family $\{\vp_C: C\in[\bn]^{\le k}\setminus\{\emptyset\},\ \min(C)=\min(A)\}$
it follows that $S_k(A,r)$ and $S_k(B,s)$ are independent if $A,B\in[\bn]^{\le k}\setminus\{\emptyset\}$, with $\min(A)\not=\min(B)$.
\end{proof}

We finish our  list of facts with a crucial observation which will allow us to prove our embedding result by induction. Assume that $k\ge 1$ and
fix some $j\in\bn$, and put $N_j= \{j+1,j+2,\ldots \}$. For $A\in[ N_j]^{k-1}$ and $r=\sum_{i=1}^n \sigma_i 2^{-i} \in \bb_{n}$, where $n:=|A|$,  (thus $(A,r)\in V_{k-1}^\omega$), for $1\le i\le n$ we put
$$T^{(i;j,+)}_{k-1}(A,r)=T_{k}^{i+1}(\{j\}\cup A,{\scriptstyle\frac{1+r}{2}}) \text{ and }T^{(i;j,-)}_{k-1}(A,r)=T_{k}^{i+1}(\{j\}\cup A,{\scriptstyle\frac{r}{2}}).$$
Note that $\frac{1+r}2=\frac12+\sum_{i=1}^n \sigma_i 2^{-(1+i)}$ and $\frac{r}2=\sum_{i=1}^n \sigma_i 2^{-(1+i)}$, thus $T^{(i;j,+)}_{k-1}(A,r)$ and $T^{(i;j,-)}_{k-1}(A,r)$ are well defined, and we put

\begin{equation*}
S^{(j,+)}_{k-1}(A,r):=\bigcup_{1\le i\le |A|, \sigma_{i}=1} T^{(i;j,+)}_{k-1}(A,r),
\end{equation*}

and 

\begin{equation*}
S^{(j,-)}_{k-1}(A,r):=\bigcup_{1\le i\le |A|, \sigma_{i}=1}T^{(i;j,-)}_{k-1}(A,r).
\end{equation*}

We define $\vpt_A= \vp_A|_{\{\vp_{\{j\}}=-1\}}$ for $A\kin[N_j]^{k-1}$. Then $(\vpt_A)_{A\in[N_j]^{k-1}}$  are independent Bernoulli random variables on the probability space $(\tilde \bp_j,\tilde \Sigma_j, \tilde\Omega_j) $, with 
$\tilde \Omega_j=\{\vp_{\{j\}}\keq-1\}$, $\tilde \Sigma_j=\Sigma\cap \{\vp_{\{j\}}=-1\}$, and for a measurable $S\subset \tilde \Omega_j$,
$\tilde \bp_j(S)= \bp(S\ |\ \{\vp_{\{j\}}=-1\})= 2\bp(S)$.
 For $A\in[N_j]^{k-1}$ and $r=\sum_{i=1}^n\sigma_i 2^{-i}\in\bb_{n}$, where $n=|A|$, and $i\le n$, with $\sigma_i=1$, we note
 that 
 \begin{align*}
T^{(i;j,+)}_{k-1}(A,r)&=T_{k}^{i+1}(\{j\}\cup A,{\scriptstyle\frac{1+r}{2}}) \\
&=  \{\vp_{\{j\}}=-1\}\cap\{ \vp_{\{j\}\cup A|_{i+1}}=1\}\cap \bigcap_{m=2, \sigma_{m-1}=1}^{i} \{ \vp_{\{j\}\cup A |_m}=-1\} \\
 &\qquad\qquad \cap \bigcap_{m=2, \sigma_{m-1}=0}^{i} \{ \vp_{\{j\}\cup A |_m}=1\} \\
  & =  \{\vpt_{A|_i}=1\}\cap \bigcap_{m<i, \sigma_{m}=1} \{ \vpt_{A |_m}=-1\} \cap \bigcap_{m<i, \sigma_{m}=0} \{ \vpt_{A |_m}=1\}.
 \end{align*}
 So we deduce that the sets $S_{k-1}^{(j,+)}(A,r)$ are defined as the sets $S_k(A,r)$ where we replace $k$ by $k-1$ and the family $(\vp_A)_{A\in[\bn]^{\le k}}$ by the family $(\vpt_A)_{A\in[N_j]^{\le k-1}}$.

\medskip
  
Similarly we can define the probability space  $(\hat\bp_j,\hat\Sigma_j, \hat\Omega_j) $, with $\hat\Omega_j=\{\vp_{\{j\}}\keq1\}$, $\hat\Sigma_j=\Sigma\cap \{\vp_{\{j\}}=1\}$, and $\hat\bp_j(S)= \bp(S\ |\ \{\vp_{\{j\}}=1\})= 2\bp(S)$ for a measurable $S\subset \hat \Omega_j$, and on it the random variables
$\vph_A= \vp_A|_{\{\vp_{\{j\}}=1\}}$ for $A\kin[N_j]^{k-1}$. Then $(\vph_A)_{A\in[N_j]^{k-1}}$ are independent Bernoulli random variables with respect to $\hat\bp_j$. We deduce as before that the sets $S_{k-1}^{(j,-)}(A,r)$ are defined as the sets $S_k(A,r)$ where we replace $k$ by $k-1$ and the family $(\vp_A)_{A\in[\bn]^{\le k}}$ by the family $(\vph_A)_{A\in[N_j]^{\le k-1}}$. This last observation will make it possible to show the following result by a straightforward induction on $k\in\bn_0$.

\begin{theorem}\label{T:3:2} For every $k\kin\bn_0$ there exists $\Psi_k: D_k^\omega\to L_1[0,1]$, such that
if $x$ and $y$ belong to the same vertical path then
\begin{equation}\label{E:3.2_7}
\big\| \Psi_k(x)-\Psi_k(y)\big\|= d_k(x,y),
\end{equation}
and if $x$ and $y$ do not belong to the same vertical path then 
\begin{equation}\label{E:3.2_8}
\frac{d_k(x,y)}{2} \le \big\| \Psi_k(x)-\Psi_k(y)\big\|\le  d_k(x,y).
\end{equation}
\end{theorem}

\begin{proof} For $k\in\bn_0$ and $x\in V_k^\omega$ let $S_k(x)\subset[0,1]$ be defined as in Lemma \ref{L:3.1}, and define the map $\Psi_k: V_k^\omega\to L_1[0,1],\quad x\mapsto 2^k\chi_{S_k(x)}.$

\medskip

For $k=0$, our claim is trivially true. Assume that our claim holds for $k\in\bn\cup\{0\}$. The equality \eqref{E:3.2_7} follows immediately from \eqref{E:3.2_1} and \eqref{E:3.2_6}, and the fact \eqref{E:2.3_3}, shown in Section \ref{S:2.3},
that 
$d_k\big((A,r),(B,s)\big)=2^{k} |s-r|$, if $(A,r),(B,s)$, lie in the same vertical path.  For general $(A,r)$ and $(B,s)$ in $V^\omega_{k+1}$ we proceed as follows. Note first  that we can assume that $(A,r), (B,s)\in V_{k+1}^\omega\setminus \{b^\omega_{k+1},t^\omega_{k+1}\}$, otherwise they would belong to the same vertical path. First let us assume that $(A,r)$ and $(B,s)$ are in $V^{(j,+)}_{k+1}$ for some $j\in\bn$. 
 Thus, write $r$ and $s$ as $r=\sum_{i=1}^{\#A} \sigma_i 2^{-i}\in\bb_{\#A}$ and  $s=\sum_{i=1}^{\#B} \tau_i 2^{-i}\in\bb_{\#B}$.
It follows that we can write $A=\{j\}\cup A'$ and $B=\{j\} \cup B'$, with   $A',B'\in[\bn]^{\le k}$, and $r=\frac{r'+1}2$ and  $s=\frac{s'+1}2$ with $r'=\sum_{i=1}^{\#A'} \sigma'_i 2^{-i}\in \bb_{\#A'}$ and  $s'=\sum_{i=1}^{\#A'} \tau'_i 2^{-i}\in \bb_{\#B'}$. Note that $\sigma_i'=\sigma_{i+1}$, for $i=1,2,\ldots, \#A'$ and  $\tau'_i=\tau_{i+1}$, for $i=1,2,\ldots, \#B'$. It follows therefore from the definition of $S_{k+1}(A,r) $ in \eqref{E:3.2_4}  and the definition of the $T_{k+1}^i(A,s)$, $i=1,2,\ldots, \#A$, with $\sigma_i=1$, in \eqref{E:3.2_3} that 
\begin{align*}
S_{k+1}(A,r)&=\big\{\vp_{\{j\}}=1\big\}\cup  \bigcup_{2\le i\le \#A, \sigma_i=1} T_{k+1}^i(A,r)\\
&=\big\{\vp_{\{j\}}=1\big\}\cup  \bigcup_{1\le i\le \#A', \sigma'_{i}=1} T_{k+1}^{i+1}(\{j\}\cup A',{\scriptstyle \frac{r'+1}2}) \\ 
& =\big\{\vp_{\{j\}}=1\big\}\cup S^{(j,+)}_k(A',r')
\intertext{and, similarly}
S_{k+1}(B,s)&= \{\vp_{\{j\}}=1\} \cup S^{(j,+)}_k(B',s'),             
\end{align*}
where the sets $S^{(j,+)}_k(A',r')$, and $S^{(j,+)}_k(B',s')$ have been defined before the statement of the theorem. It follows from the prior  observations and the induction hypothesis (using the probability space $(\tilde\Omega_j, \tilde\Sigma_j,\tilde\bp_j)$ as defined above) that 
\begin{align*} 
2^{k+1}\big\|\chi_{S_{k+1}(A,r)}- \chi_{S_{k+1}(B,s)}\big\|_{L_1[0,1]}&=2^{k+1} \big\|\chi_{S^{(j,+)}_{k}(A',r')}- \chi_{S_{k}^{(j,+)}(B',s')}\big\|_{L_1[0,1]}\\
&=2^k \big\|\chi_{S^{(j,+)}_{k}(A',r')}- \chi_{S_{k}^{(j,+)}(B',s')}\big\|_{L_1(\tilde\bp_j)},
\end{align*}
but
\begin{equation*}
\frac{d_k\big((A',r'),(B',s') \big)}{2}\le 2^k \big\|\chi_{S^{(j,+)}_{k}(A',r')}- \chi_{S_{k}^{(j,+)}(B',s')}\big\|_{L_1(\tilde\bp_j)}   \le d_k\big((A',r'),(B',s') \big).
\end{equation*}

Using now \eqref{E:2.3_10} we conclude that
\begin{equation*}
\frac{d_{k+1}\big((A,r),(B,s) \big)}{2} \le \big\| \Psi_{k+1}\big((A,r)\big)-\Psi_k\big((B,s)\big)\big\|_{L_1[0,1]}\le  d_{k+1}\big((A,r),(B,s) \big).
\end{equation*} \label{E:1.2}

If $(A,r)$ and $(B,s)$ are in $V^{(j,-)}_{k+1}$ for some $j\in\bn$, we deduce our claim similarly as in the previous case.

\medskip

If $(A,r)\in V^{(j,-)}_{k+1}$ and $(B,s)\in V^{(j,+)}_{k+1}$, for some $j\in\bn$ (or vice versa), then $(A,r)$ and $(B,s)$ lie on the same vertical path, a case we already handled.

\medskip

If $(A,r)\in V^{(i,+)}_{k+1}$ and $(B,s)\in V^{(j,+)}_{k+1}$, for $i\neq j$,   we can assume that $r\ge s> \frac12$ and  it follows from \eqref{E:2.2_1} and \eqref{E:3.2_2} that 
\begin{align*}
\| \chi_{S_{k+1}(A,r)} -\chi_{S_{k+1}(B,s)}&\|_{L_1[0,1]}=
 \| \chi_{S^c_{k+1}(A,r)} -\chi_{S^c_{k+1}(B,s)}\|_{L_1[0,1]}\\
 = \bp\big(S^c_{k+1}(A,r)\!\setminus&\!S^c_{k+1}(B,s)\big)+ \bp\big(S^c_{k+1}(B,s)\!\setminus\!S^c_{k+1}(A,r)\big) \\
 = (1-r) +(1-&s) -2(1-r)(1-s) \\
 \ge  (1-r) +(1-&s)-(1-r)\\ 
=1-s \hskip 1.4cm&\\
\ge \frac12  2^{-(k+1)}d_{k+1}&\big((A,r),(B,s)\big),&
\end{align*}
and secondly
\begin{align*}
 \big\| \chi_{S_{k+1}(A,r)} -\chi_{S_{k+1}(B,s)}\big\|_{L_1[0,1]}&=  \big\|\chi_{S^c_{k+1}(A,r)} -\chi_{S^c_{k+1}(B,s)}\big\|_{L_1[0,1]}\\
  & \le \bp(S^c_{k+1}(A,r))+\bp(S^c_{k+1}(B,s))\\
  & =2^{-(k+1)} d_{k+1}\big((A,r),(B,s)\big),
\end{align*}
which implies our claim in that case.

\medskip

If $(A,r)\in V^{(i,-)}_{k+1}$ and $(B,s)\in V^{(j,-)}_{k+1}$, for $i\neq j$,   we can proceed in a similar way as in the previous case.
  
\medskip

If $(A,r)\in V^{(i,-)}_k$ and $(B,s)\in V^{(j,+)}_k$ for $i\neq j$, we first assume that $r+s\le 1$ and thus by \eqref{E:2.3_8}, \eqref{E:2.2_1}, and \eqref{E:3.2_2}
\begin{equation*}
\big\| \chi_{S_{k+1}(A,r)} -\chi_{S_{k+1}(B,s)}\big\|_{L_1[0,1]}= r+s-2rs\le r+s= \frac{d_{k+1}\big((A,r),(B,s)\big)}{2^{(k+1)}}.
\end{equation*}
Secondly the geometric-arithmetic mean inequality and the assumption that $r+s\le 1$ we deduce that
$$r^{1/2}s^{1/2}\le \frac{r+s}2\le\frac{\sqrt{r+s}}2,$$
and thus $2rs \le \frac{r+s}2$ which implies by \eqref{E:2.3_8} that 
$$r+s-2rs\ge \frac{r+s}2=\frac12  2^{-(k+1)}d_{k+1}\big((A,r),(B,s)\big),$$
 which implies our claim.

\medskip

If $(A,r)\in V^{(i,-)}_{k+1}$ and $(B,s)\in V^{(j,+)}_{k+1}$ for $i\neq j$, and $r+s\ge 1$, we can proceed similarly, which finishes our induction step and the proof of our theorem. 
\end{proof}

\subsection{Embeddability into $L_p(Y)$}\label{S:3.3}

Recall that $D_{k}^2$ stands for the $2$-branching diamond of depth $k$. In this section it is shown that one can build an embedding of $D_k^\omega$ into $L_p(Y)$ out of an embedding of $D_k^2$ into $Y$. This ``embedding transfer'' is useful in regards of some renorming problems that will be discussed in Section \ref{S:6}.

\begin{theorem}\label{T:3.2} If for every $k\in\bn$, $D_k^2$ embeds into $Y$ with distortion at most $C\ge 1$ then for every $k\in\bn$, $D_k^\omega$ embeds into $L_p([0,1],Y)$ with distortion at most $2^{1/p}C$. More precisely, for each $k \ge 1$, there exists a $1$-Lipschitz map $\bar{\phi}_k\colon D_k^{\omega} \to L_p([0,1],Y)$ such that whenever $x$ and $y$  belong to the same vertical path in $D_{k}^{\omega}$
\begin{equation}\label{E:3.3_1}
\|\bar{\phi}_k(x) - \bar{\phi}_k(y)\|_p \stackrel{}{\ge} \frac1C d_k(x,y), 
\end{equation} 
and whenever $x$ and $y$ do not belong to the same vertical path in $D_{k}^{\omega}$
\begin{equation}\label{E:3.3_2}
\|\bar{\phi}_k(x) - \bar{\phi}_k(y)\|_p \stackrel{}{\ge} \frac{1}{2^{1/p}C}.
\end{equation} 

\end{theorem}

\begin{proof}
Let $\|\cdot\|_p$ denotes the norm in $L_p([0,1],Y)$ and $b_k^2$ (resp. $t_k^2$) denote the bottom vertex (resp. top vertex) of $D_{k}^2$. Assume, as we may, that, for each $k\ge1$, there exists $\phi_k \colon D_{k}^{2} \rightarrow Y$ that satisfies for every $x,y\in D_{k}^{2}$,
\begin{equation} \label{E:3.3_3}
\frac{1}{C} d_k(x,y) \stackrel{(37a)}{\le} \|\phi_k(x) - \phi_k(y)\|_Y \stackrel{(37b)}{\le} d_k(x,y). 
\end{equation} 
 We shall construct inductively, for each $k \ge 1$, a $1$-Lipschitz mapping $\bar{\phi}_k \colon D_k^{\omega} \rightarrow L_p([0,1],Y)$ satisfying the following conditions:

for every $x$ and $y$ in $D_{k}^{\omega}$ belonging to the same vertical path
\begin{equation} \label{E:3.3_4} \frac{1}{C} d_k(x,y) \le \|\bar{\phi}_k(x) - \bar{\phi}_k(y)\|_p,
\end{equation} 
for every $x$ and $y$ in $D_{k}^{\omega}$ belonging to different vertical paths
\begin{equation} \label{E:3.3_5} \frac{1}{2^{1/p}C} d_k(x,y) \le \|\bar{\phi}_k(x) - \bar{\phi}_k(y)\|_p,
\end{equation} 

for every $x \in D_k^\omega$ and $u\in [0,1]$, there exists $r\in D_k^2$ satisfying
\begin{equation}\label{E:3.3_6}
\bar{\phi}_k(x)(u)=\phi_k(r) \text{ and } d_k(b_k^\omega,x)=d_k(b_k^2,r).
\end{equation} 

\medskip
One more time, we shall give ourselves once and for all a sequence of independent Bernoulli random variables defined on $[0,1]$. Each time we will have to choose countably many independent Bernoulli random variables defined on $[0,1]$ we will assume, as we may, that we will leave out countably infinitely many independent Bernoulli random variables from this sequence.

\medskip

The initialization case goes as follows. Note that $D_1^{2}$ has four vertices: $t_1^2$ (top), $b_1^2$ (bottom), $v_l$ (left vertex), and $v_r$ (right vertex). Suppose we have a mapping
$\phi_1 \colon D_1^{2} \rightarrow Y$ satisfying \eqref{E:3.3_3}. Note that $D_1^{\omega}$ consists of a top vertex $t_1^\omega$, a bottom vertex $b_1^\omega$, and the sequence $(m_j)_{j\in\bn}$ of midpoints of $t_1^\omega$ and $b_1^\omega$. Let $A:= \{\varepsilon_{j} \colon j\in\bn\}$ be a countably infinite collection of independent Bernoulli random variables defined on $[0,1]$. Define $\bar{\phi}_1 \colon D_1^{\omega} \rightarrow L_p([0,1],Y)$ as follows
\begin{equation}\label{E:3.3_7}
\bar{\phi}_1(b_1^\omega) = \phi_1(b_1^2)\cdot \chi_{[0,1]},\ \bar{\phi}_1(t_1^\omega) = \phi_1(t_1^2)\cdot\chi_{[0,1]},
\end{equation} 

and 
\begin{equation}\label{E:3.3_8}
\bar{\phi}_1(m_j) = \phi_1(v_l)\cdot \chi_{\{\vp_{j}=0\}} + \phi_1(v_r)\cdot \chi_{\{\vp_{j}=1\}}. 
\end{equation}

Condition \eqref{E:3.3_6} is clearly statisfied and using (37b) it is easily checked that the map is $1$-Lipschitz. An appeal to $(37a)$ will show that 
\begin{equation*}
\frac{1}{C}d_1(x,y) \le \|\bar{\phi}_1(x) - \bar{\phi}_1(y)\|_p,
\end{equation*}
 holds whenever $x$ and $y$ belong to the same vertical path in $D_1^\omega$, which proves \eqref{E:3.3_4}. If $x=m_i$ and $y=m_j$, for some $i<j$, then elementary computations show that  
\begin{align*}
\|\bar{\phi}_1(m_i)-\bar{\phi}_1(m_j)\|_p^p&\ge \bp(\vp_{i} \ne \vp_{j})\|\phi_1(v_l)-\phi_1(v_r)\|_p^p\\
&= \frac{1}{2}\|\phi_1(v_l)-\phi_1(v_r)\|_p^p\\
&\ge \frac{1}{2C^p}d_1(v_l,v_r)^p
= \frac{1}{2C^p}d_1(m_i,m_j)^p, 
\end{align*}
which gives \eqref{E:3.3_5}.

\medskip

To carry out the inductive step, further notation is needed. Similarly to the notation for the countably branching diamond graph, let  $D^{(l,+)}_{k}$, $D^{(r,+)}_{k}$, $D^{(l,-)}_{k}$, and $D^{(r,-)}_{k}$ denote the ``top left'', ``top right'', ``bottom left'', and ``bottom right'' subdiamonds of $D_{k+1}^{2}$ of depth $k$. We identify each of them with $D_k^{2}$ via the canonical isometry which preserves ``top'', ``bottom'', ``left'' and ``right'' for all subdiamonds. Let $\phi_{k+1}$ satisfy \eqref{E:3.3_3} (where $k$ is substituted with $k+1$) and let $\phi^{(l,+)}_{k}$, $\phi^{(r,+)}_{k}$, $\phi^{(l,-)}_{k}$, and $\phi^{(r,-)}_{k}$ be the restrictions of $\phi_{k+1}$ to $D^{(l,+)}_{k+1}$, $D^{(r,+)}_{k+1}$, $D^{(l,-)}_{k+1}$, and $D^{(r,-)}_{k+1}$ respectively. Via the identification described above we regard $\phi^{(l,+)}_{k}$, $\phi^{(r,+)}_{k}$, $\phi^{(l,-)}_{k}$, and $\phi^{(r,-)}_{k}$ as mappings from $D_{k}^{2}$ into $Y$ satisfying \eqref{E:3.3_3}. By the induction hypothesis, there are maps $\bar{\phi}^{(l,+)}_{k}$, $\bar{\phi}^{(r,+)}_{k}$, $\bar{\phi}^{(l,-)}_{k}$, and $\bar{\phi}^{(r,-)}_{k}$ from $D_{k}^{\omega}$ into $L_p([0,1],Y)$ satisfying \eqref{E:3.3_4} and \eqref{E:3.3_5}. Let $(m_j)_{j\in\bn}$ be the midpoints of $t_{k+1}^\omega$ and $b_{k+1}^\omega$ in $D_{k+1}^\omega$ and take a sequence $(\varepsilon_{j})_{j\in\bn}$ of independent random Bernoulli variables on $[0,1]$. Define $\bar{\phi}_{k+1}\colon D_{k+1}^\omega\to L_p([0,1],Y)$ as follows:  
\begin{equation*}
 \bar{\phi}_{k+1}(x):=\begin{cases} \bar{\phi}^{(l,+)}_{k}(x)\cdot \chi_{\{\vp_{j}=0\}}+\bar{\phi}^{(r,+)}_{k}(x)\cdot \chi_{\{\vp_{j}=1\}}, \text{ if }x \in V_{k+1}^{(j,+)} \text{ for some }$j$,\\
\bar{\phi}^{(l,-)}_{k}(x)\cdot \chi_{\{\vp_{j}=0\}}+\bar{\phi}^{(r,-)}_{k}(x)\cdot \chi_{\{\vp_{j}=1\}}, \text{ if }x \in V_{k+1}^{(j,-)} \text{ for some }$j$.
\end{cases}
\end{equation*}
Note that $\bar{\phi}^{(l,+)}_{k}(x)$ (resp. $\bar{\phi}^{(l,+)}_{k}(x)$, $\bar{\phi}^{(l,+)}_{k}(x)$, $\bar{\phi}^{(l,+)}_{k}(x)$, $\bar{\phi}^{(l,+)}_{k}(x)$) are well-defined when we identify, as we may, $V_{k+1}^{(j,+)}$, and $V_{k+1}^{(j,-)}$ with $D_k^{\omega}$. Moreover, $\bar{\phi}_{k+1}$ is well-defined as the two formulas coincide for $x=m_j$. Condition \eqref{E:3.3_6} is clearly satisfied. We now verify \eqref{E:3.3_4}-\eqref{E:3.3_5} and the $1$-Lipschitz condition for $\bar{\phi}_{k+1}$. 

\begin{description}
\item[Case 1] If $x,y\in V_{k}^{(j,+)}$ or $x,y\in V_{k}^{(j,-)}$ for some $j\in\bn$. We only treat the case $x,y\in V_{k}^{(j,+)}$ since the case $x,y\in V_{k}^{(j,-)}$ is completely similar.
\begin{align}\label{E:3.3_9}
 \|\bar{\phi}_{k+1}&(x)-\bar{\phi}_{k+1}(y)\|_p^p\\
 &=\frac{1}{2}(\|\bar{\phi}^{(l,+)}_{k}(x)-\bar{\phi}^{(l,+)}_{k}(y)\|_p^p
+\|\bar{\phi}^{(r,+)}_{k}(x)-\bar{\phi}^{(r,+)}_{k}(y)\|_p^p). \notag
\end{align}
It now follows easily from \eqref{E:3.3_9} and the induction hypothesis, that if $x$ and $y$ belong to the same vertical path then
\begin{equation*} \label{42} \frac{1}{C} d_k(x,y) \le \|\bar{\phi}_k(x) - \bar{\phi}_k(y)\|_p,
\end{equation*}
and if $x$ and $y$ belong to different vertical paths
\begin{equation*} \label{43} \frac{1}{2^{1/p}C} d_k(x,y) \le \|\bar{\phi}_k(x) - \bar{\phi}_k(y)\|_p.
\end{equation*} 
Moreover, if $x$ and $y$ form a pair of  adjacent vertices in $D_{k+1}^{\omega}$ then either $x,y\in V_{k}^{(j,+)}$ or $x,y\in V_{k}^{(j,-)}$, and combining (37b) and \eqref{E:3.3_9} one gets
$$\|\bar{\phi}_{k+1}(x)-\bar{\phi}_{k+1}(y)\|_p \le d_{k+1}(x,y)=1,$$
which is sufficient to show that $\bar{\phi}_{k+1}$ is $1$-Lipschitz since the domain space is an unweighted graph equipped with the shortest path metric.\\

\item[Case 2] If $x\in V_{k}^{(j,+)}$ and $y\in V_{k}^{(j,-)}$. In this case $x$ and $y$ belong to the same vertical path. First consider $u \in \{\vp_{j}=0\}$. By the induction hypothesis one can assume that $\bar{\phi}_{k+1}(x)(u) = \phi_{k+1}(r)$ for some $r \in D_{k+1}^{(l,+)}$, with $d_{k+1}(b_{k+1}^2,r)=d_{k+1}(b_{k+1}^\omega,x)$, and $\bar{\phi}_{k+1}(y)(u) = \phi_{k+1}(s)$ for some $s \in D_{k+1}^{(l,-)}$, with $d_{k+1}(b_{k+1}^2,s)=d_{k+1}(b_{k+1}^\omega,y)$. Then,
\begin{align*}
 \|\bar{\phi}_{k+1}(x)(u)-\bar{\phi}_{k+1}(y)(u)\|_Y
&=\|\phi_{k+1}(r)-\phi_{k+1}(s)\|_Y\\
&\ge \frac1Cd_{k+1}(r,s)\\
&= \frac1Cd_{k+1}(x,y).
\end{align*}
A similar inequality is obtained if $u \in \{\vp_{j}=1\}$. To conclude it remains to observe that 
\begin{equation*}
 \|\bar{\phi}_{k+1}(x)-\bar{\phi}_{k+1}(y)\|_p\ge \inf_{u \in [0,1]} \|\bar{\phi}_{k+1}(x)(u)-\bar{\phi}_{k+1}(y)(u)\|_Y.
\end{equation*}

\item[Case 3] If $x\in V_{k+1}^{(i,\pm)}$ and $y\in V_{k+1}^{(j,\pm)}$ for some $i\neq j\in\bn$. In this case $x$ and $y$ do not belong to the same vertical path.
There are four possible choices of $\pm$ signs, but the argument is essentially the same in each case. So let us consider the case of $x\in V_{k+1}^{(i,+)}$ and $y\in V_{k+1}^{(j,+)}$ for some $i\neq j\in\bn$. We will proceed as in Case 2 and look at pointwise inequalities. There are four subcases to consider. If $u \in \{\vp_{i}=0\}
\cap\{\vp_{j}=1\}$. By the induction hypothesis one can assume that $\bar{\phi}_{k+1}(x)(u) = \phi_{k+1}(r)$ for some $r \in D_{k+1}^{(l,+)}$, with $d_{k+1}(b_{k+1}^2,r)=d_{k+1}(b_{k+1}^\omega,x)$, and $\bar{\phi}_{k+1}(y)(u) = \phi_{k+1}(s)$ for some $s \in D_{k+1}^{(r,+)}$, with $d_{k+1}(b_{k+1}^2,s)=d_{k+1}(b_{k+1}^\omega,y)$. Using (37a), we get
\begin{align} \label{E:3.3_10}
\|\bar{\phi}_{k+1}(x)(u) - \bar{\phi}_{k+1}(y)(u)\|_Y& = \|\phi_{k+1}(r)-\phi_{k+1}(s)\|\\
 &\ge \frac{1}{C}d_{k+1}(r,s)
 =\frac{1}{C}d_{k+1}(x,y).\notag
\end{align} A similar argument shows that \eqref{E:3.3_10} also holds for $u \in \{\varepsilon_{i}=1\}
\cap\{\varepsilon_{j}=0\}$. Hence
$$\|\bar{\phi}_{k+1}(x) - \bar{\phi}_{k+1}(y)\|_p^p \ge \bp(\vp_{i} \ne \vp_{j}) \frac{d_{k+1}(x,y)^p}{C^p}=\frac{d_{k+1}(x,y)^p}{2C^p}, $$ which gives (24a).
\end{description} 

\end{proof}

\begin{corollary}\label{C:3.1} Suppose $1 \le p < \infty$, and let $Y$ be a non-superreflexive Banach space. Then, for every $\vp>0$ and $k\in\bn$,
$D_{k}^{\omega}$ admits a bi-Lipschitz embedding into $L_p([0,1],Y)$  with distortion at most $2^{1+1/p}+\vp$.
\end{corollary}

\begin{proof} By a recent result of Pisier \cite{Pisierbook}, which refines a result of Johnson and Schechtman \cite{JohnsonSchechtman2009}, for every $\vp>0$ there exist bi-Lipschitz embeddings
$\phi_k \colon D_{k}^{2} \rightarrow Y$ such that, for all $x,y \in D_{k}^{2}$, 
$$ \frac{1}{2+\vp} d_k(x,y) \le \|\phi_k(x) - \phi_k(y)\|_Y \le d_k(x,y).$$ 
The conclusion follows from Theorem \ref{T:3.2}.
\end{proof}

\section{Non-embeddability of the countably branching diamond graphs}\label{S:5}

In a celebrated unpublished article (cf. \cite{Benyamini1985}, \cite{BenyaminiLindenstrauss2000}, or \cite{Weston1993}), Enflo used an approximate midpoint argument to show that $L_1$ and $\ell_1$ are not uniformly homeomorphic. This simple, but clever and extremely useful argument, is based on the fact that in $\ell_1$ the size of the approximate midpoint set between two points is rather small whereas it is easy to find pairs of points in $L_1$ that have infinitely many exact midpoints that are far apart. The approximate midpoint argument has been reused many times since in various disguises. The non-embeddability result presented in this section follows from a new attempt to generalize the approximate midpoint argument of Enflo.

\medskip

In order to handle ``unbalanced'' graphs such as the parasol graphs, it is necessary to introduce a slight generalization of the notion of $\delta$-approximate metric midpoints. Let $(X,d_X)$ be a metric space, $x,y\in X$ and $\delta, \lambda\in(0,1)$. A metric $\lambda$-barycenter between $x$ and $y$ is any point $z\in X$ satisfying $d_X(x,z)=\lambda d_X(x,y)$ and $d_X(z,y)=(1-\lambda)d_X(x,y)$. Let us define the set of $\delta$-approximate metric $\lambda$-barycenter set between $x$ and $y$ as
$$\text{Bar}_\lambda(x,y,\delta)=\left\{z\in X: \max\left\{\frac{d_X(x,z)}{\lambda},\frac{d_X(z,y)}{1-\lambda}\right\}\le(1+\delta)d_X(x,y)\right\}.$$
When $\lambda=\frac12$ one recovers the classical notions of metric midpoint and $\delta$-approximate midpoint set (which will be denoted by $\text{Mid}(x,y,\delta)$ in the sequel). In the Banach space setting, the two notions are related in the following elementary way.

\begin{lemma}\label{L:5.1} Let $X$ be a Banach space.
Let $\delta,\lambda \in(0,1)$. Then for every $x\in X$,
\begin{align*}
\text{Bar}_\lambda(-\lambda x,(1-\lambda)x,\delta)\subset \text{Mid}(-\max\{\lambda, 1-\lambda\} x,\max\{\lambda, 1-\lambda\} x,\delta).
\end{align*}
\end{lemma}
\begin{proof}
Let $\mu:=\max\{\lambda, 1-\lambda\}$. Assume that $z\in \text{Bar}_\lambda(-\lambda x,(1-\lambda)x,\delta)$. Since $\|-\lambda x-(1-\lambda)x\|=\lambda\|x\|$ one has that $\|z+\lambda x\|\le(1+\delta)\lambda\|x\|$ and $\|z-(1-\lambda) x\|\le(1+\delta)(1-\lambda)\|x\|$. Consider first, the case where $\mu=\lambda$. Since $\|\lambda x-(-\lambda x)\|=2\lambda\|x\|$, by definition of the $\delta$-approximate midpoint set one simply needs to show that $\|z+\lambda x\|\le (1+\delta)\lambda\|x\|$ and $\|z-\lambda x\|\le (1+\delta)\lambda\|x\|$, and the former inequality holds. For the latter inequality, since $\lambda\ge \frac12$, 
\begin{align*}
(2\lambda-1)\|x\|&=\|(1-\lambda)x-\lambda x\|\\
&\ge \|z-\lambda x\|-\|z-(1-\lambda)x\|\\
&\ge \|z-\lambda x\|-(1+\delta)(1-\lambda)\|x\|.
\end{align*}
Therefore, $\|z-\lambda x\|\le (2\lambda-1)\|x\|+(1+\delta)(1-\lambda)\|x\|\le(1+\delta)\lambda\|x\|$.

The case where $\mu=1-\lambda$ can be handled in a similar manner. Indeed, Since $\|(1-\lambda) x+(1-\lambda) x\|=2(1-\lambda)\|x\|$, it remains to show that $\|z+(1-\lambda) x\|\le (1+\delta)(1-\lambda)\|x\|$ and $\|z-(1-\lambda) x\|\le (1+\delta)(1-\lambda)\|x\|$, and the latter inequality holds. For the former inequality, since $\lambda\le \frac12$, 
\begin{align*}
(1-2\lambda)\|x\|&=\|(1-\lambda)x-\lambda x\|\\
&\ge \|z+(1-\lambda)x\|-\|z+\lambda x\|\\
&\ge \|z+(1-\lambda) x\|-(1+\delta)\lambda\|x\|.
\end{align*}
Therefore, $\|z+(1-\lambda) x\|\le (1-2\lambda)\|x\|+(1+\delta)\lambda\|x\|\le(1+\delta)(1-\lambda)\|x\|$.
\end{proof}

It is well-known that the size of a $\delta$-approximate metric midpoint set in an asymptotically uniformly convex Banach spaces is ``small''. By ``small'' we mean that the set is included in the (Banach space) sum of a compact set and a ball of small radius. We shall see that $\delta$-approximate metric $\lambda$-barycenter sets in asymptotically midpoint uniformly convex Banach spaces are also ``small''. Being asymptotically midpoint uniformly convexifiable is formerly weaker than being asymptotically uniformly convexifiable. However, it is still open whether asymptotic uniform convexity and asymptotic midpoint uniform convexity are equivalent notions up to renorming. As a by-product of our work we give one more insight on this problem (cf. Section \ref{S:6}).

\medskip

Let $Y$ be a Banach space and $t\in(0,1)$. Define
$$\tilde{\delta}_Y(t):=\inf_{y\in S_Y}\sup_{Z\in\cof(Y)}\inf_{z\in S_Z}\max\{\|y+tz\|, \|y-tz\|\}-1.$$
The norm of $Y$ is said to be asymptotically midpoint uniformly convex if $\tilde{\delta}_Y(t)>0$ for every $t\in(0,1)$. A characterization of asymptotic midpoint uniformly convex norms was given in \cite{DKRRZ} in terms of the Kuratowski measure of noncompactness of approximate midpoint sets. Recall that the Kuratowski measure of noncompactness of a subset $S$ of a metric space, denoted by $\alpha(S)$, is defined as the infimum of all $\vp>0$ such that $S$ can be covered by a finite number of sets of diameter less than $\vp$. Note that it is a property of the metric.

In \cite{DKRRZ} it was shown that a Banach space $Y$ is asymptotically midpoint uniformly convex if and only if $$\ds\lim_{\delta\rightarrow0}\sup_{y\in S_Y}\alpha(\text{Mid}(-y,y,\delta))=0.$$ To prove the main result of this section the following quantitative formulation of the ``only if'' part is needed, and the following lemma can be extracted from the proof of $ii)\implies i)$ in Theorem 2.1 from \cite{DKRRZ}.

\begin{lemma}\label{L:5.2}
If the norm of a Banach space $Y$ is asymptotically midpoint uniformly convex then for every $t\in(0,1)$ and every $y\in S_Y$ one has $$\alpha\left(\text{Mid}\left(-y,y,\tilde{\delta}_Y(t)/4\right)\right)<4t.$$
\end{lemma} 

The following lemma about smallness of $\delta$-approximate metric $\lambda$-barycenter sets in asymptotically midpoint uniformly convex spaces, now follows easily.
 
\begin{lemma}\label{L:5.3}If the norm of a Banach space $Y$ is asymptotically midpoint uniformly convex then for every $\lambda\in(0,1)$, every $t\in(0,1)$ and every $x,y\in Y$ there exists a finite subset $S$ of $Y$ such that
$$\text{Bar}_\lambda(x,y,\tilde{\delta}_Y(t)/4)\subset S+4t\max\{\lambda, 1-\lambda\}\|x-y\|B_Y.$$
\end{lemma}

\begin{proof} Let $\delta:=\tilde{\delta}_Y(t)/4$ and $\mu:=\max\{\lambda, 1-\lambda\}$. Assume as we may that $x\neq y$. Observe first that 
\begin{align*}
\text{Bar}_\lambda(x,y,\delta)=(1-\lambda)x+\lambda y+\text{Bar}_\lambda(-\lambda(y-x),(1-\lambda)(y-x),\delta).
\end{align*}
By Lemma \ref{L:5.1}
\begin{align*}
\text{Bar}_\lambda(x,y,\delta)&\subseteq (1-\lambda)x+\lambda y+\text{Mid}(-\mu(y-x),\mu(y-x),\delta)\\
&\subseteq (1-\lambda)x+\lambda y+\mu\|x-y\|\text{Mid}\left(-\frac{y-x}{\|x-y\|},\frac{y-x}{\|x-y\|},\delta\right).
\end{align*} 
The conclusion follows from Lemma \ref{L:5.2} and the definition of the Kuratowski measure of non-compactness.
\end{proof}
We now describe a certain family $\cG^\omega$ of sequences of graphs that contains the following sequences of graphs: the countably branching diamond graphs $(D_k^{\omega})_{k\in\bn}$, the countably branching Laakso graphs $(L_k^{\omega})_{k\in\bn}$, and the countably branching parasol graphs $(P_k^{\omega})_{k\in\bn}$.

A bundle with finite height is a (possibly infinite) connected graph with distinguished nodes $s$ and $t$ such that all simple $s$-$t$ paths have equal finite length. The distance between the two terminals is called the height. The simplest of all infinite (without multiple edges) bundle with finite height is the countably branching diamond graph of depth $1$. A bundle $G$ with finite height has the particular property that for all vertex $x\in G$, $d_G(s,x)+d_G(x,t)=d_G(s,t):=h(G)$. 

\begin{defn}
Let $\cG^\omega$ be the family of all sequences of graphs $(G_k^\omega)_{k\in\bn}$ satisfying the following requirements:

\begin{enumerate}[i)]
\item The base (directed) graph $G_1^\omega$ is an infinite bundle with finite height.
\item $G_{k+1}^\omega:=G_{k}^\omega \oslash G_1^\omega$, for $k\ge 1$.  
\end{enumerate}
\end{defn}
The graphs $G_{k}^\omega$ constructed above are infinite bundles with height $h(G_1^\omega)^k$.
The sequence of countably branching diamond graphs is clearly obtained by taking the base graph to be the infinite bundle $D_1^\omega$, while for the countably branching Laakso graphs and the countably branching parasol graphs the base graphs are respectively the infinite bundles $P_1^\omega$ and $L_1^\omega$ depicted below. For the non-symmetric graph $P_1^\omega$ we assume that the edges are directed from the bottom of the parasol to its tip.
 
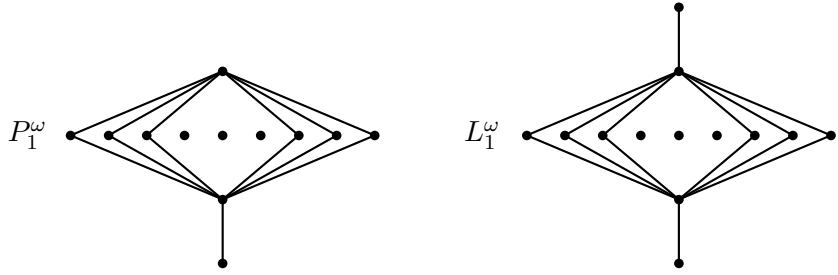
\begin{pspicture}
\rput(0,3){

\uput[0](0,2){$P_1^\omega$}
					\psdots(3,3)
\psdots(1,2)(1.5,2)(2,2)(2.5,2)(3,2)(3.5,2)(4,2)(4.5,2)(5,2) 
					\psdots(3,1)
					\psdots(3,0)
\psline(3,3)(1,2)(1,2)(3,1)
\psline(3,3)(1.5,2)(1.5,2)(3,1)
\psline(3,3)(2,2)(2,2)(3,1)
\psline(3,3)(4,2)(4,2)(3,1)
\psline(3,3)(4.5,2)(4.5,2)(3,1)
\psline(3,3)(5,2)(5,2)(3,1)

\psline(3,1)(3,0)

\uput[0](6,2){$L_1^\omega$}
					\psdots(9,4)
					\psdots(9,3)
\psdots(7,2)(7.5,2)(8,2)(8.5,2)(9,2)(9.5,2)(10,2)(10.5,2)(11,2) 
					\psdots(9,1)
					\psdots(9,0)
\psline(9,4)(9,3)
\psline(9,3)(7,2)(7,2)(9,1)
\psline(9,3)(7.5,2)(7.5,2)(9,1)
\psline(9,3)(8,2)(8,2)(9,1)
\psline(9,3)(10,2)(10,2)(9,1)
\psline(9,3)(10.5,2)(10.5,2)(9,1)
\psline(9,3)(11,2)(11,2)(9,1)
\psline(9,1)(9,0)
}					
\end{pspicture}

\begin{lemma}\label{L:5.4}
Let $G_1^\omega$ be an infinite bundle whose terminal vertices are $v_b$ and $v_t$. Let $Y$ be a Banach space whose norm is asymptotically midpoint uniformly convex. If $f:G_1^{\omega}\to Y$ is bi-Lipschitz embedding with distortion $C$. Then, there exists $\rho:=\rho(G_1^\omega)>0$ such that,
\begin{equation}\label{E:5_1}
\|f(v_t)-f(v_b)\|<\text{Lip}(f)\left(1-\frac{1}{5}\tilde{\delta}_Y\left(\frac{1}{9\rho C}\right)\right)d_{G_1^\omega}(v_t,v_b).
\end{equation}
\end{lemma}

\begin{proof}
Without loss of generality assume that $\text{Lip}(f)=1$. Let  $h:=d_{G_1^\omega}(v_b,v_t)$ denote the height of the bundle. Since the bundle is infinite with finite height by the pigeonhole principle there exists a sequence of vertices $(v_i)_{i\in\bn}$ such that $d(v_b,v_i)+d(v_i,v_t)=h$, and  $k=d(v_b,v_i)$ for all $i\in\bn$. Let $\rho:=\max\{k,h-k\}$ and $\lambda:=1-\frac{k}{h}$. Inequality \eqref{E:5_1} follows from the following claim.

\begin{claim}\label{C:5.1}
There exists $j\in\mathbb{N}$ such that 
\begin{equation}\label{E:5_2}
f(v_j)\notin\text{Bar}_\lambda\left(f(v_t),f(v_b),\frac{1}{4}\tilde{\delta}_Y\left(\frac{1}{9\rho C}\right)\right).\end{equation}
\end{claim}

Indeed, assuming the claim we have either
\begin{equation*}
\|f(v_j)-f(v_t)\|>\lambda\left(1+\frac{1}{4}\tilde{\delta}_Y\left(\frac{1}{9\rho C}\right)\right)\|f(v_t)-f(v_b)\|
\end{equation*}
or
\begin{equation*}
\|f(v_j)-f(v_b)\|>(1-\lambda)\left(1+\frac{1}{4}\tilde{\delta}_Y\left(\frac{1}{9\rho C}\right)\right)\|f(v_t)-f(v_b)\|,
\end{equation*}
which implies in both cases
\begin{align*}
\|f(v_t)-f(v_b)\|& <d(v_t,v_b)\left(1+\frac{1}{4}\tilde{\delta}_X\left(\frac{1}{9\rho C}\right)\right)^{-1}\\
 & \le d(v_t,v_b)\left(1-\frac{1}{5}\tilde{\delta}_X\left(\frac{1}{9\rho C}\right)\right).
 \end{align*}

It remains to prove the claim. By Lemma \ref{L:5.3} there exists a finite subset $S:=\{s_1,\dots,s_n\}\subset Y$ such that 
\begin{equation*}
\text{Bar}_\lambda\left(f(v_t),f(v_b),\frac{1}{4}\tilde{\delta}_Y\left(\frac{1}{9\rho C}\right)\right)\subset S+\frac{4}{9\rho C}\max\{\lambda, 1-\lambda\}\|f(v_t)-f(v_b)\|B_Y.
\end{equation*}
If for every $i$, 
\begin{equation*}
f(v_i)\in\text{Bar}_\lambda\left(f(v_t),f(v_b),\frac{1}{4}\tilde{\delta}_Y\left(\frac{1}{9\rho C}\right)\right)
\end{equation*}
then $f(v_i)=s_{n_i}+y_i$ with $s_{n_i}\in S$ and $y_i\in Y$ so that 
\begin{equation*}
\|y_i\|\le\frac{4}{9\rho C}\max\{\lambda, 1-\lambda\}\|f(v_t)-f(v_b)\|.
\end{equation*}
However, for $i\neq j$ one has

\begin{align*}
\|s_{n_i}-s_{n_j}\|&\ge\|f(v_i)-f(v_j)\|-\|y_i-y_j\|\\
&\ge \frac{d_{G_1^\omega}(v_i,v_j)}{C}-\frac{8}{9\rho C}\max\{\lambda, 1-\lambda\}\|f(v_t)-f(v_b)\|\\
&\ge \frac{1}{C}-\frac{8}{9\rho C}\max\{\lambda, 1-\lambda\}d_{G_1^\omega}(v_t,v_b)\\
&\ge \frac{1}{C}-\frac{8}{9C}>0, 
\end{align*}
which contradicts the fact that $S$ is finite.

\end{proof}

The next proposition is the self-improvement argument \`a la Johnson and Schechtman adapted to our setting.
\begin{proposition}\label{P:5.1}Let $(G_k^\omega)_{k\in\bn}\in \cG^\omega$, and $Y$ be an asymptotically midpoint uniformly convex Banach space. There exists $\rho>0$, such that if $k\in\bn$ and  $G_k^\omega$ embeds bi-Lipschitzly into $Y$ with distortion $C$, then $G_{k-1}^\omega$ embeds bi-Lipschitzly into $Y$ with distortion at most $C\left(1-\frac{1}{5}\tilde{\delta}_Y(\frac{1}{9\rho C})\right)$.
\end{proposition}

\begin{proof}
Let $Y$ be a Banach space with an asymptotically midpoint uniformly convex norm. The argument is similar to the proof of Proposition $2.1$ in \cite{BaudierZhang2016} and the formal argument using set-theoretic representations of the graphs shall be simply sketched. Let $f_k$ be a bi-Lipschitz  embedding of $G_k^\omega$ into $Y$ that is non-contracting and $C$-Lipschitz. Note that the subset of vertices $V(G_{k-1}^\omega)\subset V(G_{k}^\omega):=V(G_{k-1}^\omega\oslash G_1^\omega)$ forms an isometric copy (up to a scaling factor $s:=h(G_1^\omega)$) of $G_{k-1}^\omega$. Define $g_k$ to be the restriction of the embedding $f_k$ to the subset $V(G_{k-1}^\omega)\subset V(G_{k}^\omega):=V(G_{k-1}^\omega\oslash G_1^\omega)$ rescaled by a the factor $s$. Since in our setting it suffices to check the distortion on pair of adjacent vertices, it follows from Lemma \ref{L:5.4} that $g_k$ is a bi-Lipschitz embedding with distortion at most $C\left(1-\frac{1}{5}\tilde{\delta}_Y(\frac{1}{9\rho C})\right)$.

\end{proof}

Our last theorem is an asymptotic version of a result of Johnson and Schechtman in \cite{JohnsonSchechtman2009}.

\begin{theorem}\label{T:5.1}Let $(G_k^\omega)_{k\in\bn}\in \cG^\omega$.\\
If $Y$ is a Banach space admitting an equivalent asymptotically midpoint uniformly convex norm, then $\sup_{k\in\bn}c_Y(G^\omega_k)=\infty$. In particular, if the equivalent asymptotically midpoint uniformly convex norm has power type $p\in(1,\infty)$ then $c_Y(G^\omega_k)\gtrsim k^{1/p}$.
\end{theorem}

\begin{proof}
Let $C_k:=c_Y(G^\omega_k)$ and assume without loss of generality that $Y$ is a Banach space with an asymptotically midpoint uniformly convex norm. Assume that $C:=\sup_{k\ge1}C_k<\infty$. It follows from Proposition \ref{P:5.1} and the monotonicity of the modulus that
\begin{equation*}
C_{k-1}\le C_k\left(1-\frac{1}{5}\tilde{\delta}_Y\left(\frac{1}{9\rho C_k}\right)\right)\le C_k\left(1-\frac{1}{5}\tilde{\delta}_Y\left(\frac{1}{9\rho C}\right)\right).
\end{equation*}
Letting $k$ go to infinity gives a contradiction.

If moreover there is a constant $\gamma>0$ such that the modulus of asymptotic midpoint uniform convexity of $Y$ satisfies $\tilde{\delta}_Y(t)\ge\gamma t^p$ for some $p\in(1,\infty)$, then by Proposition \ref{P:5.1}
$$C_{k-1}\le C_k\left(1-\frac{1}{5}\tilde{\delta}_Y\left(\frac{1}{9\rho C_k}\right)\right)
\le C_k\left(1-\frac{\gamma}{5\cdot(9\rho C_k)^p}\right).$$
Therefore $C_k-C_{k-1}\ge\frac{K}{C_k^{p-1}}$, where $K=\frac{\gamma}{5\cdot(9\rho)^p}$, and hence
$$C_k\ge\sum_{j=2}^k\frac{K}{C_{j}^{p-1}}+C_1\ge\frac{K(k-1)}{C_{k}^{p-1}}.$$ The conclusion follows easily.
\end{proof}

Since property ($\beta$) of Rolewicz implies asymptotic uniform convexity and hence asymptotic midpoint uniform convexity, Theorem \ref{T:5.1} is a generalization of a result in \cite{BaudierZhang2016} saying that if $Y$ is a Banach space that has an equivalent norm with property $(\beta)$ of Rolewicz, then $\sup_{k\in\bn}c_Y(P^\omega_k)=\infty$ and $\sup_{k\in\bn}c_Y(L^\omega_k)=\infty$. Since there are (reflexive or not) Banach spaces that are asymptotically uniformly convexifiable but not asymptotically uniformly smoothable, neither the
sequences $(L^\omega_k)_{k\in\bn}$ nor $(P^\omega_k)_{k\in\bn}$ are sequences of test spaces for the class of $(\beta)$-renormable spaces. This gives a negative answer to Problem $4.1$ in \cite{BaudierZhang2016}.

\section{Applications}\label{S:6}

In this last section, the main applications of our work are gathered. The first application deals with the embeddability of the countably branching diamond graphs into certain Banach spaces. 

\begin{corollary}\label{C:6.1}\ \\
\begin{enumerate}[i)]
\item There is a bi-Lipschitz embedding $f_k$ of $D_k^\omega$ into $\co^+$ with distortion at most $3$ such that for all $x\in D_k^\omega$, $|\supp(f_k(x))|\le k+1$.
\item  Let $k\in\bn$ and let $\varepsilon>0$. If $p_k\ge \frac{\ln(2k+2)}{\ln(1+\varepsilon/3)}$, then $D_k^\omega$ admits a bi-Lipschitz embedding into $\ell_{p_k}$ with distortion at most $3+\varepsilon$.
\item Let $(p_n)_{n\in\bn}\subset(1,\infty)$ such that $\lim_{n\to\infty} p_n=\infty$, and let $Y:=(\sum_{n=1}^\infty \ell_{p_n})_{\ell_2}$, then $c_Y(D_k^\omega)\le 3$ for all $k\in\bn$.

 \end{enumerate}
\end{corollary} 

\begin{proof} 
\begin{enumerate}[i)]
\item It follows from Theorem \ref{T:3.3} and Example \ref{Ex:1}.
\item It follows from $i)$ above that the support of $f_k(x)-f_k(y)$ has size at most $2k+2$, and hence
$$\frac{\|f_k(x) - f_k(y)\|_{p_k}}{(1+ \varepsilon/3)}  \le \|f_k(x) - f_k(y)\|_\infty \le d_k(x,y)$$ and $$d_k(x,y) \le 3\|f_k(x)-f_k(y)\|_\infty \le 3 \|f_k(x)-f_k(y)\|_{p_k} $$
\item it follows clearly from the previous item.
\end{enumerate} 
\end{proof}

The next corollary, where tight estimates for the $L_p$-spaces distortion of the countably branching diamond graphs are obtained, complements the results of Lee and Naor in \cite {LeeNaor2004} for the $2$-branching diamond graphs.
 \begin{corollary}\ \\
 
 \begin{enumerate}[i)]
 \item For $1 \le p < \infty$, $c_{\ell_p}(D_k^{\omega})\approx k^{1/p}$.
 \item  For $1 < p < \infty$, $c_{L_p}(D_k^{\omega})\approx \min\{k^{1/p}, \sqrt{k}\}$.
 \end{enumerate}
 \end{corollary}

 \begin{proof} 
 \begin{enumerate}
 \item The upper estimate simply follows from Corollary \ref{C:6.1} (1) and H\"older's inequality. The lower estimate for the distortion follows from the fact that the modulus of asymptotic uniform convexity satisfies $\overline{\delta}_{\ell_p}(\varepsilon)\gtrsim \varepsilon^p$. 
\item Since $L_p[0,1]$ contains a linearly isometric copy of $\ell_2$, the upper estimate follows again from Corollary \ref{C:6.1} (1) and H\"older's inequality.. For $1<p<2$, the lower estimate was proved for $D_{k}^2$ by Lee and Naor \cite{LeeNaor2004}.
For $2 \le p < \infty$, the lower estimate follows from the fact that the modulus of asymptotic uniform convexity satisfies  $\overline{\delta}_{L_p}(\varepsilon)\gtrsim \varepsilon^p$.
\end{enumerate}\end{proof}

Our third application can be easily derived by combining Theorem \ref{T:3.3}, Theorem \ref{T:4.1} and Theorem \ref{T:5.1}, and is the most important one. Indeed, within the class of reflexive Banach space with an unconditional asymptotic structure it resolves the metric characterization problem in terms of graph preclusion for the class of asymptotically uniformly convexifiable spaces. This metric characterization is only the third of this type in the asymptotic Ribe program after the Baudier-Kalton-Lancien characterization \cite{BKL2010} and the Motakis-Schlumprecht characterization \cite{MotakisSchlumprecht}. Note that we can omit the separability assumption since it follows from \cite{DKLR}, that every (non-separable) reflexive Banach space $Y$ with an unconditional asymptotic structure and $\Sz(Y^*)>\omega$ has a separable subspace with the same properties.

\begin{corollary}\label{C:6.3} Let $Y$ be a reflexive Banach space with an unconditional asymptotic structure. Then,

$$Y\in\langle AUC\rangle\textrm{ if and only if } \sup_{k\in\bn} c_{Y}(D_k^\omega)=\infty.$$
More precisely, the sequence $(D_k^\omega)_{k\in\bn}$ is a uniformly characterizing sequence for the class of asymptotically uniformly convexifiable spaces within the class of reflexive Banach spaces with an unconditional asymptotic structure.
\end{corollary} 

The two last applications are in renorming theory. It was proven in \cite{DKRRZ}, under the assumption that $Y$
has an unconditional basis (but without assuming reflexivity), that $Y\in\langle AMUC\rangle$ if and only if $Y\in\langle AUC\rangle.$ Whether this equivalence holds in full generality is still open. A simple consequence of Theorem \ref{T:5.1} and Corollary \ref{C:6.3} is that the equivalence holds for separable reflexive Banach space with an unconditional asymptotically structure.
\begin{corollary}\label{C:6.4} Let $Y$ be a separable reflexive Banach space with an unconditional asymptotic structure. Then 
$$ Y\in\langle AMUC\rangle\textrm{ if and only if }Y\in\langle AUC\rangle.$$
\end{corollary}


Also, as a consequence of Corollary \ref{C:3.1} and the fact that the 2-branching diamond graphs form a sequence of test-spaces for the class uniformly convexifiable Banach spaces \cite{JohnsonSchechtman2009} one obtains:  
\begin{corollary} let $Y$ be a Banach space and let $1<p<\infty$. Then,
$$L_p([0,1],Y)\in\langle AMUC\rangle \text{ if and only if } Y \in\langle UC\rangle.$$
\end{corollary} 

It is worth noticing that there exist reflexive Banach spaces of type $2$ that are not super-reflexive (\cite{James1978}, \cite{PisierXu1987}). Let $Y$ be such a space, it follows that $L_2(Y)$ has type $2$ but $\sup_{k\in\bn}c_{L_2(Y)}(D_k^{\omega})<\infty$. 

\medskip

\textbf{Acknowledgments:} 
The authors wish to thank the organizers of the Workshop in Analysis and Probability, W. B. Johnson, D. Kerr, and G. Pisier. Most of this work was carried out when the third and fourth authors where participants of the workshop in July 2015 and July 2016.

\end{document}